\documentclass[review]{elsarticle}
\usepackage{graphicx}
\usepackage{lipsum}
\usepackage{algorithmic}
  \usepackage{geometry}
\usepackage{epstopdf}
 \usepackage{mathrsfs}
  \usepackage{booktabs,longtable}
   \usepackage{indentfirst}
\usepackage{epsfig}
\usepackage{lineno,hyperref}
\usepackage{bookmark}
\usepackage{multirow}
 \usepackage{amsfonts}
\usepackage{multirow}
 \usepackage{algorithm}
 \usepackage[figuresright]{rotating}
 \usepackage{amscd}
\usepackage{mathrsfs}
\usepackage{subfigure}
\usepackage{amstext}
\usepackage{amssymb}
\usepackage{fancyhdr}
\usepackage{amsmath}
\usepackage{color}
\usepackage{cleveref}
\usepackage{threeparttable}
\modulolinenumbers[5]
\biboptions{sort&compress}
\newtheorem{theorem}{Theorem}
\newtheorem{lemma}{Lemma}
\newtheorem{corollary}{Corollary}
\newdefinition{remark}{Remark}
\newproof{proof}{Proof}

\journal{Journal of \LaTeX\ Templates}

\begin{document}

\begin{frontmatter}

\title{Randomized block subsampling Kaczmarz-Motzkin method\tnoteref{mytitlenote}}
\tnotetext[mytitlenote]{The work is supported by the National Natural Science Foundation of China (No. 11671060) and the Natural Science Foundation of Chongqing, China (No. cstc2019jcyj-msxmX0267)}

\author{Yanjun Zhang}
\author{Hanyu Li\corref{mycor}}
\cortext[mycor]{Corresponding author}
\ead{lihy.hy@gmail.com or hyli@cqu.edu.cn}

\address{College of Mathematics and Statistics, Chongqing University, Chongqing 401331, P.R. China}
\begin{abstract}

 By introducing a subsampling strategy, we propose a randomized block Kaczmarz-Motzkin method for solving linear systems. Such strategy not only determines the block size, but also combines and extends two famous strategies, i.e.,  randomness and greed, and hence can inherit their advantages. Theoretical analysis shows that the proposed method converges linearly in expectation to the least-Euclidean-norm solution. Several numerical examples are reported to verify the efficiency and feasibility of the new method.
\end{abstract}

\begin{keyword}
Kaczmarz method; Motzkin method; block updating; subsampling strategy
\MSC[2010] 65F10, 65F20
\end{keyword}

\end{frontmatter}


\section{Introduction}
\label{sec;introduction}
Consider the consistent linear system
\begin{align}
Ax=b,  \label{Sec11}
\end{align}
where $A\in \mathbb{R}^{m\times n}$ with $m\geq n $, $b\in \mathbb{R}^{m}$, and $x$ is an unknown vector. A popular method for solving the problem is the Kaczmarz method \cite{kaczmarz1}, which is an effective row-action iteration solver. Due to its simplicity, speediness, and low memory usage, various variants including randomized versions \cite{Strohmer2009,Completion2015,gower2015randomized}, acceleration versions \cite{Eldar2011,liu2016accelerated,jiao2017preasymptotic} and greedy versions \cite{de2017sampling,bai2018greedy,gower2021adaptive} have been widely discussed and studied over the past decades. Besides, in many computer architectures, the block Kaczmarz iterative updating is more efficient than the simple Kaczmarz method. This is because the former adopts a subset of rows of the coefficient matrix in each iteration while the latter only utilizes one row. 
So, how to select a good set of indices and devise an efficient block method is also an active topic. 

In 2014, Needell and Tropp \cite{needell2014paved} constructed a randomized block Kaczmarz (RBK) method, which 
selects an index subset $\tau$ uniformly at random from $[m]=\{1, 2, \cdots, m\}$. 
There are two weaknesses of the RBK method. 
One is that it is possible to select the same index subset twice consecutively and then result in no progress in the second iteration, and the other one 
is that this method requires a good row paving, which is difficult to achieve when the range of row norms of $A$ is large. In addition, there may be a lot of extra costs associated with building good paving.

Subsequently, many greedy-type block algorithms \cite{Niu2020,liu2021greedy,zhang2021block,Zhang2022MK} were 
proposed to make up for the two shortcomings mentioned above. A large number of numerical experiments in the literature verify that the greedy versions indeed have faster convergence speed. However, all these greedy block methods require at least one full pass over the coefficient matrix at each iteration, which leads to 
high per-iteration costs in terms of memory and computation when the system is large. Moreover, 
some of them, like the GBK method in \cite{Niu2020}, the BEM method in \cite{liu2021greedy}, the BSKM1 method in \cite{zhang2021block}, and the GMBK method in \cite{Zhang2022MK}, cannot control the size of the index subset in each iteration, which may result in extreme cases. That is, it is possible to have either only one index or all indices in the subset. This 
is inconsistent with the philosophy of the block iteration.

In this paper, we introduce a subsampling strategy which combines the randomized and greedy block strategies discussed above, and construct a 
randomized block subsampling Kaczmarz-Motzkin (RB-SKM) method. 
Specifically, the RB-SKM method first selects an index set from $[m]$ randomly like the RBK method, and then uses the greedy strategy to determine the final iterative index subset with definite block size within this randomized sampled constraints. So, compared to the existing block methods, our method has three obvious advantages at the same time: the randomness indicates a very cheap 
cost per iteration, the greed implies a fast convergence rate and the determined block size means the effectiveness of the block iteration.

The rest of this paper is organized as follows. In Section \ref{sec: RB-SKM}, we propose the RB-SKM method, give its convergence analysis, and discuss the relationships between our method and the existing ones.  Section \ref{sec:experiments} is devoted to 
the experimental results. Finally, we conclude this paper with some remarks.

\section{The RB-SKM method}
\label{sec: RB-SKM}

Before presenting our method, we first provide some preparations. 
For a matrix $A=(A_{(i, j)})\in \mathbb{R}^{m\times n}$, $A^{(i)}$, $\text{range}(A)$, $\text{nnz}(A)$, $A^{\dag}$, and $A_{\tau}$ denote its $i$th row, column space, nonzero elements counts, Moore-Penrose pseudoinverse, and the restriction onto the row indices in the set $\tau$, respectively. We let the positive eigenvalues of $A^{T}A$ be always arranged in algebraically nonincreasing order: $\lambda_{\max }(A^{T}A)=\lambda_{1}(A^{T}A) \geq \lambda_{2}(A^{T}A)  \geq \cdots  \geq \lambda_{\min }^{+}(A^{T}A)>0,$ and use $|\tau|$, $ \mathbb{E}^{k}$ and $ \mathbb{E}$ to denote the number of elements of a set $\mathcal{\tau}$, the conditional expectation conditioned on the first $k$ iterations and the full expected value, respectively. In addition, the following lemma is necessary throughout the paper.
\begin{lemma}
\label{lem-1}(\cite{horn2012matrix})
Let $A \in \mathbb{R}^{n\times n}$ be symmetric and $A_t \in \mathbb{R}^{t\times t}$ be its principal submatrix. Then
$$\lambda_{n-t+i}(A) \leqslant \lambda_{i}\left(A_{t}\right) \leqslant \lambda_{i}(A), \quad i=1, \cdots, t.$$
\end{lemma}

Now, we give the RB-SKM method shown in Algorithm \ref{Al: The RB-SKM method} for solving the problem (\ref{Sec11}). Specifically, the method operates by randomly sampling an index set $\tau_k$ from $[ m ]$, computing the residual of this set, and projecting onto the constraints corresponding to the
$\delta$ largest magnitude entries of this subresidual.

\begin{algorithm}
\caption{The RB-SKM method}
\label{Al: The RB-SKM method}
\begin{algorithmic}[1]
\STATE{Input: parameters $\beta,~\delta \in [m]$, the initial estimate $x^0\in \mathbb{R}^n$ and $r^0=b-Ax^0$.}
\FOR{$k=0, 1, 2, \cdots $ until convergence,}
\STATE{Select an index set $\tau_k$ of size $\beta$ uniformly at random from $[ m ]$.
}
\STATE{Compute $\widetilde{r}_{\tau_k}=\lvert r^k_{\tau_k}\rvert$.}
\STATE{From 
the set $\tau_k$, determine the index subset $ \mathcal{I}_{k}$ of size $\delta $ satisfying $\widetilde{r}^{(i)}_{\tau_k}\geq \widetilde{r}^{(j)}_{\tau_k}$, where $i\in \mathcal{I}_{k}$ and $j\in {\tau_k \backslash\mathcal{I}_{k}}$.}
\STATE{Compute $d^k=A_{\mathcal{I}_{k}}^{\dagger}(b_{\mathcal{I}_{k}}-A_{\mathcal{I}_{k}}x^{k})$. }
\STATE{Update $x^{k+1}=x^{k}+d^k $ and $r^{k+1}=r^{k}-Ad^k$. }
\ENDFOR
\end{algorithmic}
\end{algorithm}

\begin{remark}
\label{Al_rmk1}
By varying the parameters $\beta$ and $\delta$, some existing methods can be recovered as the special cases of the RB-SKM method; see 
Table \ref{Al_rm1:tab1} for details.
 \begin{table}[]
\centering
   \fontsize{8}{8}\selectfont
       \caption{Summary of special methods of the RB-SKM method for different parameters $\beta$ and $\delta$.}
    \label{Al_rm1:tab1}
    \begin{threeparttable}
    \begin{tabular}{c c c c c c}
 \hline
Method  & BSKM1\tnote{*}~ \cite{zhang2021block}& Motzkin \cite{Motzkin54}& RBK \cite{needell2014paved} & SKM \cite{de2017sampling}
& RK\tnote{**}~~ \cite{Strohmer2009}             \cr \hline
$\beta$          & m                         & m                       & $\beta\in[m]$                        & $\beta\in[m]$
& 1     \cr
$\delta $ & $\delta\in[m]$                         & 1                      & $\beta $                      & 1
& 1                         \cr\hline
\end{tabular}
\begin{tablenotes}
\footnotesize
\item[*] Assume that the cardinality of the index subsets equals to $\delta$ per iteration.
\item[**] The randomized Kaczmarz method with equal probability.
\end{tablenotes}
\end{threeparttable}
\end{table}
\end{remark}

\begin{remark}
\label{Al_rmk2}
The updating formula of the RB-SKM method in Algorithm \ref{Al: The RB-SKM method} can also be replaced by a pseudoinverse-free version presented in \cite{necoara2019faster}, i.e.,
\begin{equation*}
x^{k+1}=x^{k}-\alpha_{k}\left(\sum_{i \in {\mathcal{I}_{k}}} w_{i} \frac{A^{(i)} x^{k}-b^{(i)}}{\left\|A^{(i)}\right\|^{2}_2} (A^{(i)})^{T}\right),
\end{equation*}
and its special variants widely discussed in \cite{du2020randomized,du2021doubly,chen2022fast}.
\end{remark}

Below, we provide the convergence analysis for the RB-SKM method.
\begin{theorem}
\label{th:RB-SKM}
From an initial guess $x^0\in{\text{range}(A^T)}$, the sequence $\{x^k\}_{k=0}^\infty$ generated by the RB-SKM method converges linearly in expectation to the least-Euclidean-norm solution $x^{\star}=A^{\dag}b$ and
\begin{align}
\mathbb{E}^{k}[\|x^{k+1}-x^\star\|^{2}_2]
\leq
\left(1- \frac{\delta}{\xi_k}\frac{\beta}{m} \frac{\lambda_{\min }^{+} (A^TA)}{\lambda_{\max} (A_{\mathcal{I}}^TA_{\mathcal{I}})}  \right)\| x^{k}-x^\star \|^{2}_2,\label{th:RB-SKM-1}
\end{align}
where
$$
\lambda_{\max}^{-1}(A_{\mathcal{I}}^TA_{\mathcal{I}})= \min\limits _{  \mathcal{I}_{k}} \lambda_{\max}^{-1}(A_{\mathcal{I}_k}^TA_{\mathcal{I}_k})
\quad \text{and}\quad
\xi_k=\frac{\sum  \limits _{\tau_k \in \binom{[m]}{\beta}} \left\|A_{\tau_k} x^{k}- b_{\tau_k}\right\|_{2}^{2}}{\sum  \limits _{\tau_k \in \binom{[m]}{\beta}} \min\limits _{i\in  \mathcal{I}_{k}}\left(A^{(i)} x^{k}- b^{(i)}\right)^{2}}.
$$
\end{theorem}
\begin{proof}
Following a similar proof process of Theorem 1 in \cite{zhang2021block}, we can obtain the inequality
\begin{align}
\|x^{k+1}-x^\star\|^{2}_2 \leq  \| x^{k}-x^\star \|^{2}_2-\lambda_{\max}^{-1}(A_{\mathcal{I}_{k}}^TA_{\mathcal{I}_{k}})\sum \limits _{i_k\in \mathcal{I}_{k}}(A^{(i_k)} x^{k}- b^{(i_k)})^2. \notag
\end{align}
Thus, according to the RB-SKM method in Algorithm \ref{Al: The RB-SKM method}, we get
\begin{align}
\|x^{k+1}-x^\star\|^{2}_2
&\leq
\| x^{k}-x^\star \|^{2}_2-\lambda_{\max}^{-1}(A_{\mathcal{I}_{k}}^TA_{\mathcal{I}_{k}})\delta\min\limits _{i\in \mathcal{I}_{k}}(A^{(i)} x^{k}- b^{(i)})^2 .\notag
\end{align}
Now, taking expectation conditioned on $x^{k}$ and considering the definition of $\lambda_{\max}^{-1}(A_{\mathcal{I}}^TA_{\mathcal{I}})$, we have
\begin{align}
\mathbb{E}^{k}\left[\|x^{k+1}-x^\star\|^{2}_2 \right]
&\leq  \| x^{k}-x^\star \|^{2}_2-\mathbb{E}^{k}\left[\lambda_{\max}^{-1}(A_{\mathcal{I}_{k}}^TA_{\mathcal{I}_{k}})\delta\min\limits _{i\in \mathcal{I}_{k}}(A^{(i)} x^{k}- b^{(i)})^2\right]\notag
\\
&\leq  \| x^{k}-x^\star \|^{2}_2-\lambda_{\max}^{-1}(A_{\mathcal{I}}^TA_{\mathcal{I}}) \delta\mathbb{E}^{k}\left[ \min\limits _{i\in\mathcal{I}_{k}}(A^{(i)} x^{k}- b^{(i)})^2\right]\notag
\\
&=  \| x^{k}-x^\star \|^{2}_2-\lambda_{\max}^{-1}(A_{\mathcal{I}}^TA_{\mathcal{I}}) \delta\sum  \limits _{\tau_k \in \binom{[m]}{\beta}} \frac{1}{\binom{m}{\beta}} \min\limits _{i\in \mathcal{I}_{k}}(A^{(i)} x^{k}- b^{(i)})^2 ,\notag
\end{align}
which together with the definition of $\xi_k$ leads to
\begin{align}
\mathbb{E}^{k}\left[\|x^{k+1}-x^\star\|^{2}_2 \right]
&\leq
\| x^{k}-x^\star \|^{2}_2-\lambda_{\max}^{-1}(A_{\mathcal{I}}^TA_{\mathcal{I}}) \delta \frac{1}{\binom{m}{\beta}}  \frac{1}{\xi_k}\sum  \limits _{\tau_k \in \binom{[m]}{\beta}} \left \|A_{\tau_k} x^{k}- b_{\tau_k}\right\|_{2}^{2}  \notag
\\
&=
\| x^{k}-x^\star \|^{2}_2-\lambda_{\max}^{-1}(A_{\mathcal{I}}^TA_{\mathcal{I}}) \delta \frac{1}{\binom{m}{\beta}}  \frac{1}{\xi_k}\frac{\binom{m}{\beta}\beta}{m}\|A x^{k}- b\|_{2}^{2}.  \notag
\end{align}
Further, 
together with the Courant-Fisher theorem:
\begin{align}
\|Ax\|^2_2\geq\lambda_{\min }^{+}\left( A^{T} A\right)\|x\|^2_2  \label{2}\ \textrm{is valid for any vector}\ x \in {\text{range}(A^T)},\notag
\end{align}
it yields
\begin{align}
\mathbb{E}^{k}\left[\|x^{k+1}-x^\star\|^{2}_2 \right]
&\leq
\left(1- \frac{\delta}{\xi_k}\frac{\beta}{m} \frac{\lambda_{\min }^{+} (A^TA)}{\lambda_{\max} (A_{\mathcal{I}}^TA_{\mathcal{I}})}  \right)\| x^{k}-x^\star \|^{2}_2. \notag
\end{align}
So, the desired result (\ref{th:RB-SKM-1}) is obtained.
\end{proof}
\begin{remark}\label{rek:RB-SKM-1}
Assuming $  \min\limits _{i\in  \mathcal{I}_{k}} \left(A^{(i)} x^{k}- b^{(i)}\right)^{2}\neq 0$, we have $\xi_k\geq\delta$. The lower bound of $\xi_k $ is achieved by setting $ \left(A^{(i)} x^{k}- b^{(i)}\right)^{2}= \left(A^{(j)} x^{k}- b^{(j)}\right)^{2}$ for any $i, j\in \mathcal{I}_{k}$ and $ \left(A^{(i)} x^{k}- b^{(i)}\right)^{2}=0$ for any $i\in[m]\backslash \mathcal{I}_{k}  $. Now, combining with the facts $\beta\leq m$ and $\lambda_{\min }^{+} (A^TA)\leq\lambda_{\max} (A ^TA )$, we get
$$
1- \frac{\delta}{\xi_k}\frac{\beta}{m} \frac{\lambda_{\min }^{+} (A^TA)}{\lambda_{\max} (A ^TA )} <1.
$$
Further, according to Lemma \ref{lem-1}, it is easy to obtain $\lambda_{\max} (A_{\mathcal{I}}^TA_{\mathcal{I}})\leq\lambda_{\max} (A ^TA )$. Then we can get
$$
\rho_{\text{RB-SKM}}=1- \frac{\delta}{\xi_k}\frac{\beta}{m} \frac{\lambda_{\min }^{+} (A^TA)}{\lambda_{\max} (A_{\mathcal{I}}^TA_{\mathcal{I}})} \leq1- \frac{\delta}{\xi_k}\frac{\beta}{m} \frac{\lambda_{\min }^{+} (A^TA)}{\lambda_{\max} (A ^TA )} <1,
$$
which means that the convergence factor of the RB-SKM method is smaller than 1.
\end{remark}

The following remarks specify the parameters in the RB-SKM method to obtain the convergence results for the specific methods shown in Table \ref{Al_rm1:tab1} and then we compare them with the existing results.

\begin{remark}[connection to BSKM1]
\label{rek:recovery BSKM1}
When $\beta=m$ and $\delta\in[m]$,  from Theorem \ref{th:RB-SKM}, we can get a convergence result of the BSKM1 method: 
\begin{equation}
\mathbb{E}^{k}[\|x^{k+1}-x^\star\|^{2}_2]
\leq
\left(1- \frac{\delta}{\xi_k} \frac{\lambda_{\min }^{+} (A^TA)}{\lambda_{\max} (A_{\mathcal{I}}^TA_{\mathcal{I}})}  \right)\| x^{k}-x^\star \|^{2}_2,\label{1000}
\end{equation}
where
$
\xi_k=\frac{\sum  \limits _{\tau_k \in \binom{[m]}{\beta}} \left\|A_{\tau_k} x^{k}- b_{\tau_k}\right\|_{2}^{2}}{\sum  \limits _{\tau_k \in \binom{[m]}{\beta}} \min\limits _{i\in  \mathcal{I}_{k}}\left(A^{(i)} x^{k}- b^{(i)}\right)^{2}} 
$. Note that the existing convergence result given in \cite[Theorem 1]{zhang2021block} for the BSKM1 method is
\begin{equation}
\mathbb{E}^{k}[\|x^{k+1}-x^\star\|^{2}_2]
\leq
\left(1- \frac{\bar{\beta}}{\gamma_k}\frac{\delta}{m} \frac{\lambda_{\min }^{+} (A^TA)}{\lambda_{\max} (A_{\mathcal{I}}^TA_{\mathcal{I}})}  \right)\| x^{k}-x^\star \|^{2}_2, \label{10001}
\end{equation}
where $\gamma_k=\frac{\sum  \limits _{\tau_k \in \binom{m}{\bar{\beta}}} \left \|A_{\tau_k} x_{k}- b_{\tau_k}\right\|_{2}^{2}}{\sum  \limits _{\tau_k \in \binom{m}{\bar{\beta}}}  \left\|A_{\tau_k} x_{k}- b_{\tau_k}\right\|_{\infty}^{2}}$ and $\bar{\beta}\leq m$.
Since $\frac{1}{\xi_k}\geq\frac{\bar{\beta}}{\gamma_k m} $ implies $1- \frac{\delta}{\xi_k} \frac{\lambda_{\min }^{+} (A^TA)}{\lambda_{\max} (A_{\mathcal{I}}^TA_{\mathcal{I}})} \leq 1- \frac{\bar{\beta}}{\gamma_k}\frac{\delta}{m} \frac{\lambda_{\min }^{+} (A^TA)}{\lambda_{\max} (A_{\mathcal{I}}^TA_{\mathcal{I}})} $, 
the new convergence result \eqref{1000} improves the existing one \eqref{10001} under the aforementioned condition. 
\end{remark}

\begin{remark}[connection to Motzkin]
\label{rek:recovery Motzkin}
If the coefficient matrix $A$ is standardized, i.e., $\| A^{(i)} \|_2^2=1$ for $i\in [m]$, then, setting $\beta=m$ and $\delta=1$, our convergence result in Theorem \ref{th:RB-SKM} reduces to
\begin{align}
\|x^{k+1}-x^\star\|^{2}_2
\leq
\left(1- \frac{\lambda_{\min }^{+} (A^TA)}{\xi_k}    \right)\| x^{k}-x^\star \|^{2}_2,\notag
\end{align}
where
$$
\xi_k=\frac{\sum  \limits _{\tau_k \in \binom{[m]}{\beta}} \left\|A_{\tau_k} x^{k}- b_{\tau_k}\right\|_{2}^{2}}{\sum  \limits _{\tau_k \in \binom{[m]}{\beta}} \min\limits _{i\in  \mathcal{I}_{k}}\left(A^{(i)} x^{k}- b^{(i)}\right)^{2}}
=\frac{\sum  \limits _{\tau_k \in \binom{[m]}{\beta}} \left\|A_{\tau_k} x^{k}- b_{\tau_k}\right\|_{2}^{2}}{\sum  \limits _{\tau_k \in \binom{[m]}{\beta}}  \left\|A_{\tau_k} x^{k}- b_{\tau_k}\right\|_{\infty}^{2}}.
$$
It is just the convergence result 
given in \cite[Remark 3 ]{haddock2021greed} for the Motzkin method.
\end{remark}

\begin{remark}[connection to RBK]
\label{rek:recovery RBK}
Setting $\beta\in[m]$ and $\delta=\beta$, if the coefficient matrix $A$ is standardized and its row partition $T=\{\tau_1, \tau_2, \cdots, \tau_p\}$ satisfies $\sum_{i=1}^{p}\lvert \tau_i\rvert=m$, $\lvert \tau_i\rvert=\beta$ and $\lambda_{\max} (A_{\tau_i}^TA_{\tau_i})\leq \beta$ for any $i\in[p]$, then, 
from Theorem \ref{th:RB-SKM}, we have
\begin{align}
\mathbb{E}^{k}[\|x^{k+1}-x^\star\|^{2}_2]
\leq
\left(1- \frac{\beta}{\xi_k}\frac{\beta}{m} \frac{\lambda_{\min }^{+} (A^TA)}{\lambda_{\max} (A_{\mathcal{I}}^TA_{\mathcal{I}})}  \right)\| x^{k}-x^\star \|^{2}_2,\notag
\end{align}
where
$$
\lambda_{\max}(A_{\mathcal{I}}^TA_{\mathcal{I}})\leq \beta
\quad \text{and}\quad
\xi_k\geq\beta.
$$
Further, we get
\begin{align}
\mathbb{E}[\|x^{k}-x^\star\|^{2}_2]
\leq
\prod_{j=0}^{j=k-1}\left(1- \frac{\beta}{\xi_j} \frac{\lambda_{\min }^{+} (A^TA)}{m}  \right)\| x^{0}-x^\star \|^{2}_2,\notag
\end{align}
which is an improved convergence result for the RBK method over the one presented in \cite[Theorem 1.2]{needell2014paved},
\begin{align}
\mathbb{E}[\|x^{k}-x^\star\|^{2}_2]
\leq
\left(1-  \frac{\lambda_{\min }^{+} (A^TA)}{\beta m}  \right)^k\| x^{0}-x^\star \|^{2}_2,\notag
\end{align}
when $\beta^2\geq \xi_j$.
\end{remark}

\begin{remark}[connection to SKM]
\label{rek:recovery SKM}
Assuming that the coefficient matrix $A$ is standardized and the parameters $\beta\in[m]$ and $\delta=1$, according to Theorem \ref{th:RB-SKM}, we have 
\begin{align}
\mathbb{E}^{k}[\|x^{k+1}-x^\star\|^{2}_2]
\leq
\left(1- \frac{\beta\lambda_{\min }^{+} (A^TA) }{\xi_k m}  \right)\| x^{k}-x^\star \|^{2}_2,\notag
\end{align}
where
$$
\xi_k=\frac{\sum  \limits _{\tau_k \in \binom{[m]}{\beta}} \left\|A_{\tau_k} x^{k}- b_{\tau_k}\right\|_{2}^{2}}{\sum  \limits _{\tau_k \in \binom{[m]}{\beta}} \min\limits _{i\in  \mathcal{I}_{k}}\left(A^{(i)} x^{k}- b^{(i)}\right)^{2}}
=\frac{\sum  \limits _{\tau_k \in \binom{[m]}{\beta}} \left\|A_{\tau_k} x^{k}- b_{\tau_k}\right\|_{2}^{2}}{\sum  \limits _{\tau_k \in \binom{[m]}{\beta}}  \left\|A_{\tau_k} x^{k}- b_{\tau_k}\right\|_{\infty}^{2}}.
$$
Further, we get
\begin{align}
\mathbb{E}[\|x^{k}-x^\star\|^{2}_2]
\leq
\prod_{j=0}^{j=k-1}\left(1- \frac{\beta\lambda_{\min }^{+} (A^TA) }{\xi_j m}  \right)\| x^{0}-x^\star \|^{2}_2,\notag
\end{align}
which is just the convergence result for the SKM method shown in \cite[Corollary 2.1]{haddock2021greed}.
\end{remark}

\begin{remark}[connection to RK]
\label{rek:recovery RK}
Suppose that the coefficient matrix $A$ is standardized and $\beta=\delta=1$. Then, from Theorem \ref{th:RB-SKM}, we have 
\begin{align}
\mathbb{E}^{k}[\|x^{k+1}-x^\star\|^{2}_2]
\leq
\left(1- \frac{1}{\xi_k}\frac{\lambda_{\min }^{+} (A^TA)}{m}   \right)\| x^{k}-x^\star \|^{2}_2,\notag
\end{align}
where
$$
\xi_k=\frac{\sum  \limits _{\tau_k \in \binom{[m]}{\beta}} \left\|A_{\tau_k} x^{k}- b_{\tau_k}\right\|_{2}^{2}}{\sum  \limits _{\tau_k \in \binom{[m]}{\beta}} \min\limits _{i\in  \mathcal{I}_{k}}\left(A^{(i)} x^{k}- b^{(i)}\right)^{2}}=1.
$$
So, we can obtain
\begin{align}
\mathbb{E}[\|x^{k}-x^\star\|^{2}_2]
\leq
\left(1-  \frac{\lambda_{\min }^{+} (A^TA)}{m}   \right)^k\| x^{0}-x^\star \|^{2}_2,\notag
\end{align}
which is just the convergence result of the RK method given in \cite[Theorem 2]{Strohmer2009}.
\end{remark}

Since $\xi_k$ plays an important role in the convergence behavior of the RB-SKM method, and we only know its lower bound from Remark \ref{rek:RB-SKM-1}, we next analyze its upper bound in expectation for the Gaussian matrix case. The proof is similar to Lemma 2.2 of \cite{haddock2019motzkin} and Proposition 4.1 of \cite{haddock2021greed}.

\begin{lemma}
\label{lem-Gaussian-case} 
Assume that $A \in \mathbb{R}^{m\times n}$ is a random Gaussian matrix with $A_{(i, j)}\thicksim\mathcal{N}(0,\sigma^2)$, $x^\star=0$ and $b=0$. For each $\tau \in \binom{[m]}{\beta}$, $ \mathcal{I}_{k}\subseteq\tau$ and $\lvert\mathcal{I}_{k}\rvert=\delta$, let $\mathcal{I}_{\tau}\subseteq\tau  \backslash\mathcal{I}_{k}$ be a set of rows that are independent of $x$ and $\lvert\mathcal{I}_{\tau}\rvert\leq\beta-\delta$. If $x$ is independent of at least $m^{\prime}$ rows of $A$, then  
\begin{align*}
 \xi_k
&=
\frac{\sum  \limits _{\tau \in \binom{[m]}{\beta}} \mathbb{E}[\left\|A_{\tau} x\right\|_{2}^{2}]}{\sum  \limits _{\tau  \in \binom{[m]}{\beta}} \mathbb{E}[ \min\limits _{i\in  \mathcal{I}_{k}}\left(A^{(i)} x \right)^{2}]}
\lesssim
\frac{\binom{m}{\beta} \left((\beta-\delta)n+\sum_{i \in \tau \backslash I_\tau}\|A^{(i)}\|_2^2/\sigma^2 \right)}{\binom{m^{\prime}}{\beta}  \log(\beta-\delta)}.
\end{align*}
\end{lemma}

\begin{proof}

Firstly, for 
the numerator, by using the Cauchy-Schwarz inequality, we have
\begin{align}
\sum  \limits _{\tau \in \binom{[m]}{\beta}} \mathbb{E}[\left\|A_{\tau} x\right\|_{2}^{2}]
&\leq \sum  \limits _{\tau \in \binom{[m]}{\beta}} \sum  \limits _{i \in \tau}\mathbb{E}[\|A^{(i)}\|_{2}^{2}\left\|  x\right\|_{2}^{2}] \notag
\\
&=\sum_{\tau \in \binom{[m]}{\beta}}\left(\sum_{i \in I_\tau} \mathbb{E}[\|A^{(i)}\|_2^2\|x\|_2^2]+\sum_{i \in \tau \backslash I_\tau}\|A^{(i)}\|_2^2\|x\|_2^2\right),\notag
\end{align}
which together with the fact $\mathbb{E}[\|A^{(i)}\|_2^2]=n\sigma^2$ leads to 
\begin{align}
\sum  \limits _{\tau \in \binom{[m]}{\beta}} \mathbb{E}[\left\|A_{\tau} x\right\|_{2}^{2}]
&\leq 
\sum_{\tau \in \binom{[m]}{\beta}}\left(\sum_{i \in I_\tau}n\sigma^2\|x\|_2^2+\sum_{i \in \tau \backslash I_\tau}\|A^{(i)}\|_2^2\|x\|_2^2\right)\notag
\\
&\leq 
\sum_{\tau \in \binom{[m]}{\beta}}\left((\beta-\delta)n\sigma^2 +\sum_{i \in \tau \backslash I_\tau}\|A^{(i)}\|_2^2\right)\|x\|_2^2\notag
\\
&\approx
\binom{m}{\beta} \left((\beta-\delta)n\sigma^2 +\sum_{i \in \tau \backslash I_\tau}\|A^{(i)}\|_2^2\right)\|x\|_2^2.\notag
\end{align}

Secondly, for the 
denominator, we have
\begin{align}
\sum  \limits _{\tau  \in \binom{[m]}{\beta}} \mathbb{E}[ \min\limits _{i\in  \mathcal{I}_{k}}\left(A^{(i)} x \right)^{2}]
&\geq
\sum  \limits _{\tau  \in \binom{[m]}{\beta}} \mathbb{E}[ \max\limits _{i\in\tau  \backslash\mathcal{I}_{k}}\left(A^{(i)} x \right)^{2}]
\geq
\sum  \limits _{\tau  \in \binom{[m]}{\beta}} \mathbb{E}[ \max\limits _{i\in \mathcal{I}_{\tau}}\left(A^{(i)} x \right)^{2}]\notag\\
&\geq
\sum  \limits _{\tau  \in \binom{[m]}{\beta}} \mathbb{E}[\left( \max\limits _{i\in\mathcal{I}_{\tau}}A^{(i)} x \right)^{2}].\notag
\end{align}
Further, by applying the Jensen's inequality, we get
\begin{align}
\sum  \limits _{\tau  \in \binom{[m]}{\beta}} \mathbb{E}[ \min\limits _{i\in  \mathcal{I}_{k}}\left(A^{(i)} x \right)^{2}]
&\geq
\sum  \limits _{\tau  \in \binom{[m]}{\beta}} \left( \mathbb{E}[\max\limits _{i\in\mathcal{I}_{\tau}}A^{(i)} x ]\right)^{2} \notag
\\
&\geq
\sum  \limits _{\tau  \in \binom{[m]}{\beta},~\lvert\mathcal{I}_{\tau}\rvert=\beta-\delta } \left( \mathbb{E}[\max\limits _{i\in\mathcal{I}_{\tau}}A^{(i)} x ]\right)^{2}.\notag
\end{align}
Moreover, from the fact $A^{(i)} x\thicksim\mathcal{N}(0,\sigma^2\|x\|_2^2)$ and the estimation for the maximum of independent normal random variables, we have 
\begin{align}
 \mathbb{E}[\max\limits _{i\in\mathcal{I}_{\tau}}A^{(i)} x ] 
 \gtrsim
 \sigma \|x\|_2 \sqrt{\log(\lvert\mathcal{I}_{\tau}\rvert)}. \notag
\end{align}
Thus,
\begin{align}
\sum  \limits _{\tau  \in \binom{[m]}{\beta}} \mathbb{E}[ \min\limits _{i\in  \mathcal{I}_{k}}\left(A^{(i)} x \right)^{2}]
& \gtrsim
\sum  \limits _{\tau  \in \binom{[m]}{\beta},~\lvert\mathcal{I}_{\tau}\rvert=\beta-\delta }\sigma^2 \|x\|_2^2\log(\beta-\delta) \notag
\\
&\geq
\binom{m^{\prime}}{\beta} \sigma^2 \|x\|_2^2\log(\beta-\delta).\notag
\end{align}

Therefore,  
\begin{align*}
 \xi_k
&=
\frac{\sum  \limits _{\tau \in \binom{[m]}{\beta}} \mathbb{E}[\left\|A_{\tau} x\right\|_{2}^{2}]}{\sum  \limits _{\tau  \in \binom{[m]}{\beta}} \mathbb{E}[ \min\limits _{i\in  \mathcal{I}_{k}}\left(A^{(i)} x \right)^{2}]}
 \lesssim
\frac{\binom{m}{\beta} \left((\beta-\delta)n\sigma^2 +\sum_{i \in \tau \backslash I_\tau}\|A^{(i)}\|_2^2\right)\|x\|_2^2}{\binom{m^{\prime}}{\beta} \sigma^2 \|x\|_2^2\log(\beta-\delta)}\\
 &=
\frac{\binom{m}{\beta} \left((\beta-\delta)n+\sum_{i \in \tau \backslash I_\tau}\|A^{(i)}\|_2^2/\sigma^2 \right)}{\binom{m^{\prime}}{\beta}  \log(\beta-\delta)},
\end{align*}
which is the desired result.
\end{proof}
\begin{remark} 
\label{rek:bound of xi}
From Lemma \ref{lem-Gaussian-case}, we know that $\xi_k$ is upper bounded by $\mathcal{O}\left(\frac{(\beta-\delta)n}{ \log(\beta-\delta)} \right)$ when $A$ is a random Gaussian matrix. Here we guess that the sharp upper bound is $\mathcal{O}\left(\frac{\beta-\delta}{ \log(\beta-\delta)} \right)$, and the factor $n$ there is due to 
the proof technique. Figure \ref{fig_for_xi} confirms the conjecture.
\begin{figure}[ht]
 \begin{center}
\includegraphics [height=4cm,width=5.5cm ]{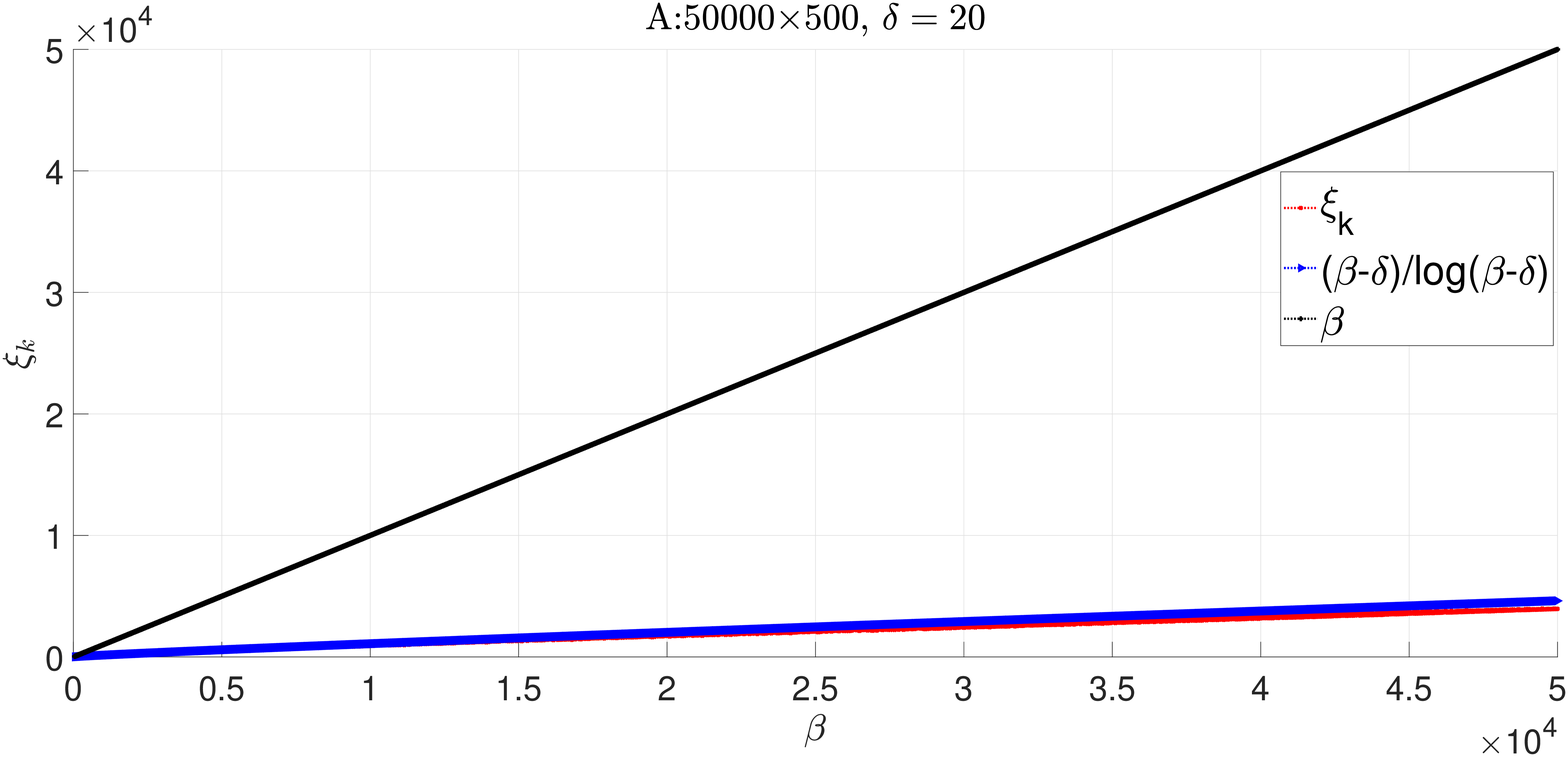}
\includegraphics [height=4cm,width=5.5cm ]{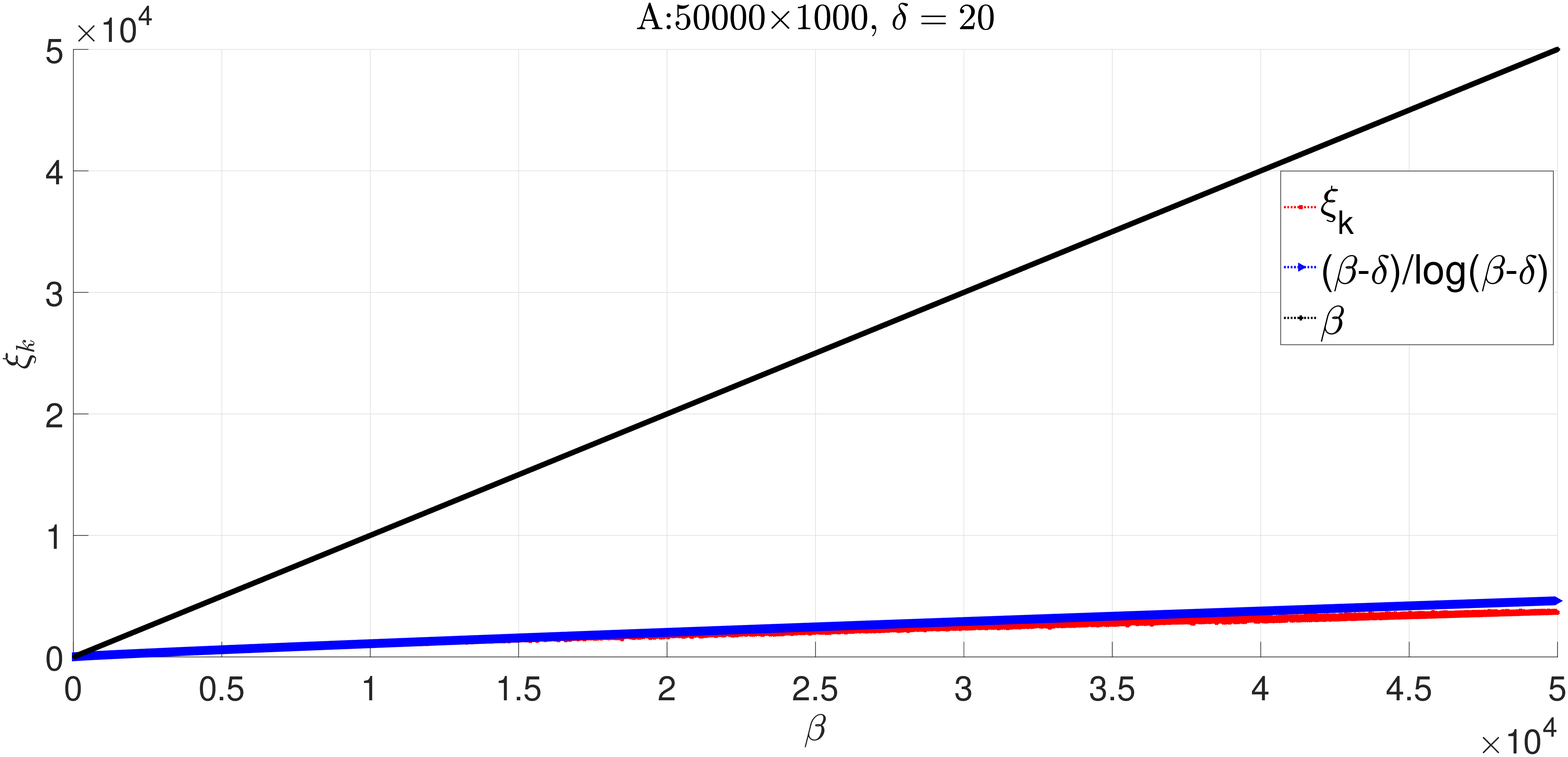}
 \end{center}
\caption{$\xi_k$ for various choices of $\beta$. Left: Gaussian matrix $A\in \mathbb{R}^{50000\times 500}$.  Right: Gaussian matrix $A\in \mathbb{R}^{50000\times 1000}$. }\label{fig_for_xi}
\end{figure}
\end{remark}

Based on the above results and our main theorem, we can get the following Corollary \ref{corollary:for th1} on the convergence for the RB-SKM
method with the coefficient matrix $A$ being the random Gaussian matrix.

\begin{corollary}
\label{corollary:for th1}
Assume that $A \in \mathbb{R}^{m\times n}$ is a random Gaussian matrix with $A_{(i, j)}\thicksim\mathcal{N}(0,\sigma^2)$, and $x^{\star}=A^{\dag}b$ is the least-Euclidean-norm solution. If the sequence $\{x^k\}_{k=0}^\infty$ generated by the RB-SKM method starts from an initial guess $x^0\in{\text{range}(A^T)}$, then we have
\begin{align}
\mathbb{E} [\|x^{k+1}-x^\star\|^{2}_2]
\lesssim
\mathbb{E} [\left(1- \frac{\log(\beta-\delta)\delta}{\beta-\delta}\frac{\beta}{m} \frac{\lambda_{\min }^{+} (A^TA)}{\lambda_{\max} (A^TA)}  \right)\| x^{k}-x^\star \|^{2}_2].\notag
\end{align}
\end{corollary}
\section{Experimental results}
\label{sec:experiments}
In this section, we first test the impact of the parameters $\beta$ and $\delta$ on the RB-SKM method, and then compare our method with the BSKM2 method from \cite{zhang2021block} in terms of the iteration numbers (denoted as ``Iteration''), operation counts and computing time in seconds (denoted as ``CPU time(s)''). All the numerical results are the arithmetic average of 10 repeated trials for each method, and all experiments start from an initial vector $x^0=0$ and terminate when the \emph{relative residual} (RR) satisfies $\mathrm{RR}=\frac{\left\|b-A x^{k}\right\|_2^2}{\left\|b-A x^{0}\right\|_2^2}<10^{-6}$, or when the number of iterations exceeds 200000.

\subsection{The impact of the parameters $\beta$ and $\delta$ on the RB-SKM method}
\label{sec:The impact of parameters}

We report in left two subgraphs of Figure \ref{fig_for different parameters} how sensitive the RB-SKM method is for variation of the parameter $\beta$ (from $\delta$ to $m$) and show in right two subgraphs of Figure \ref{fig_for different parameters} how the parameter $\delta$ (from $5$ to $\beta$) affects the RB-SKM method. In the specific experiments, the coefficient matrix $A$ is generated by the MATLAB function \texttt{sprandn(m,n,0.2,0.8)}, the solution vector $x^\star\in \mathbb{R}^{n}$ is generated by the MATLAB function \texttt{randn}, and the vector $b\in \mathbb{R}^{m}$ is generated by setting $b=Ax^\star$.

From Figure \ref{fig_for different parameters}, we find that all results generally show a similar trend, that is, the CPU times of the RB-SKM method with respect to the increase of the parameter $\beta$ or the parameter $\delta$ shows a U-shape. This phenomenon implies that a minimum for CPU time occurs for $\beta$ between $\delta$ and $m$ when the parameter $\delta$ is fixed, and for $\delta$ between $5$ and $\beta$ when the parameter $\beta$ is fixed.
\begin{figure}[ht]
 \begin{center}
\includegraphics [height=3.2cm,width=3.7cm ]{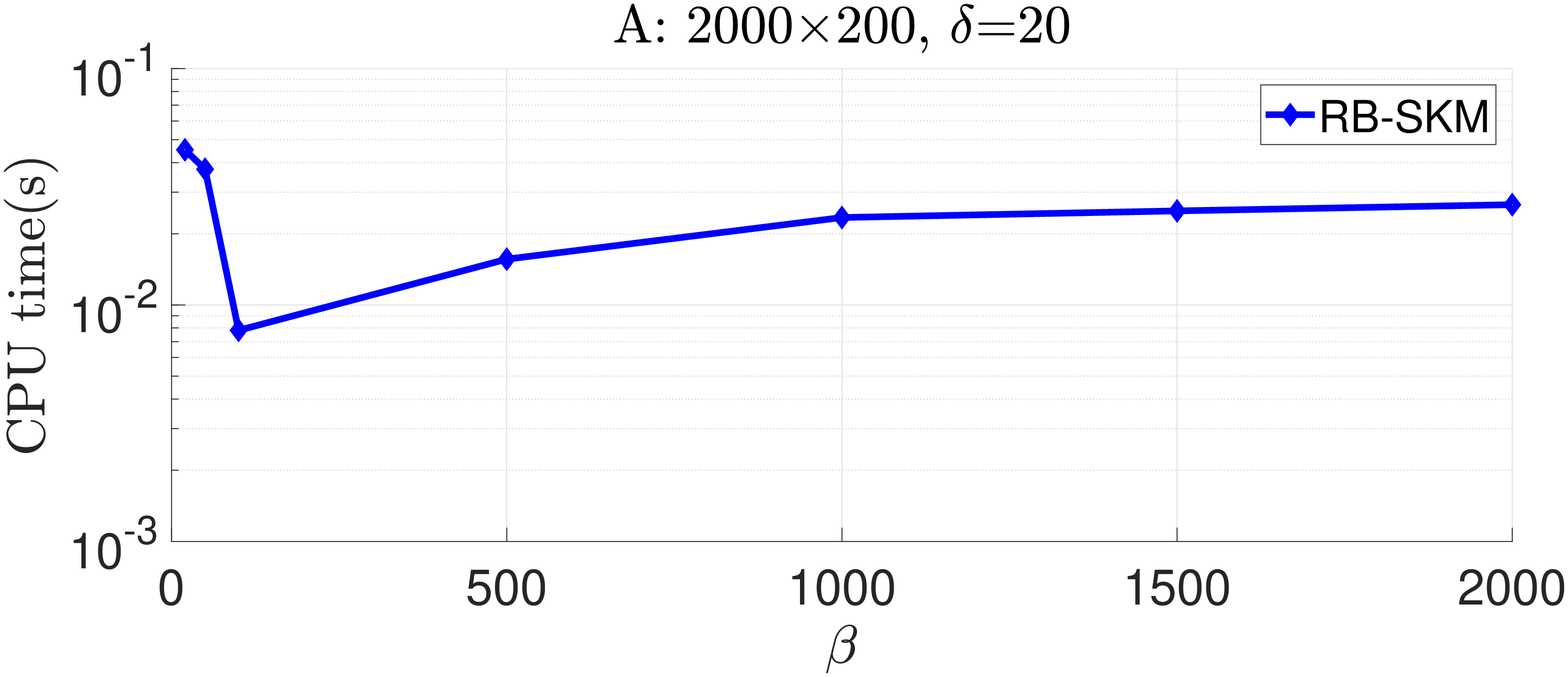}
\includegraphics [height=3.2cm,width=3.7cm ]{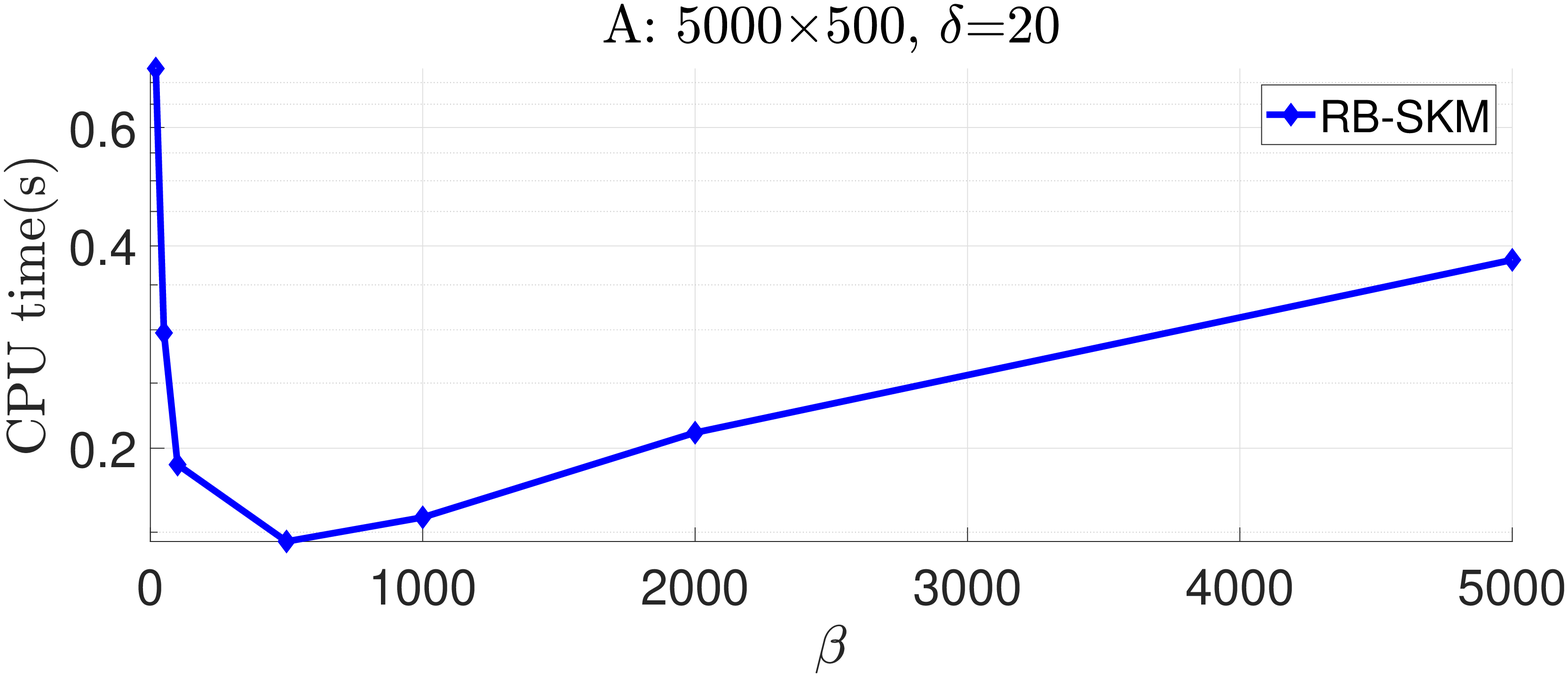}
\includegraphics [height=3.2cm,width=3.7cm ]{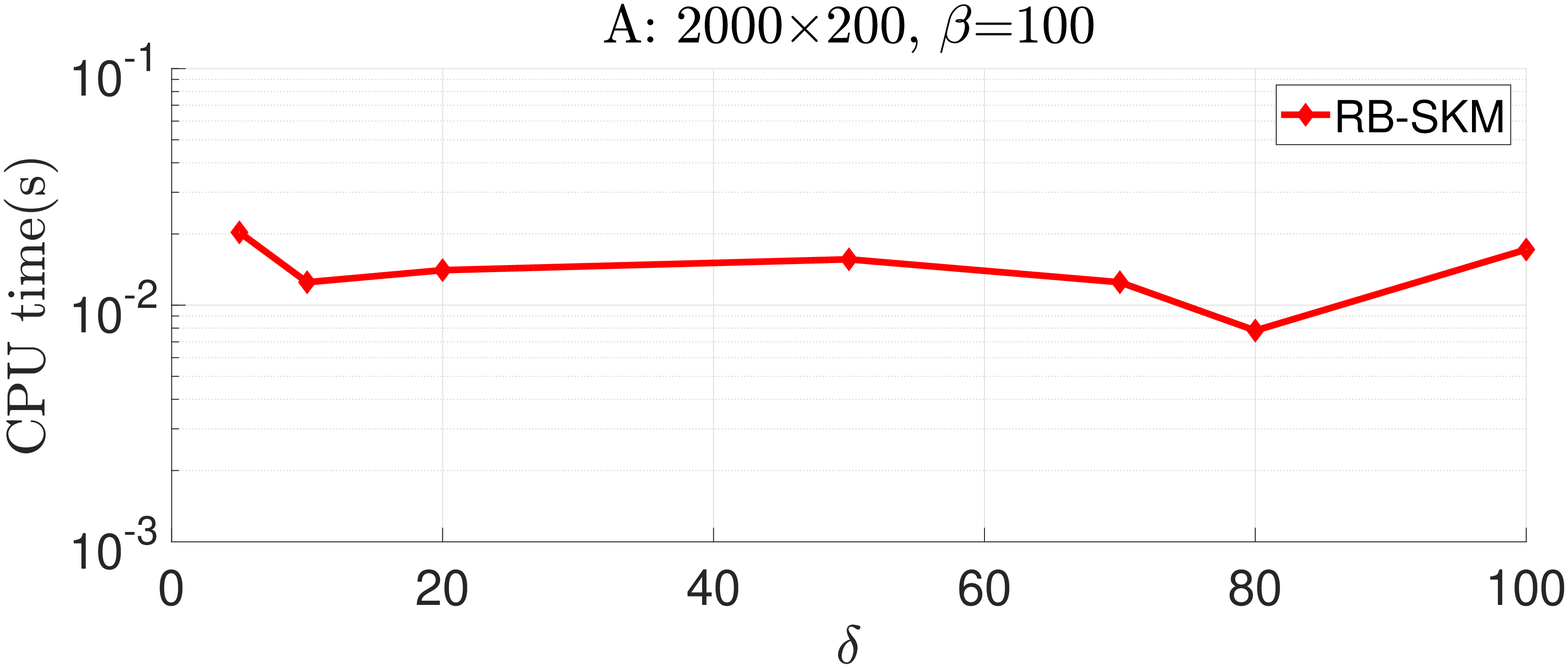}
\includegraphics [height=3.2cm,width=3.7cm ]{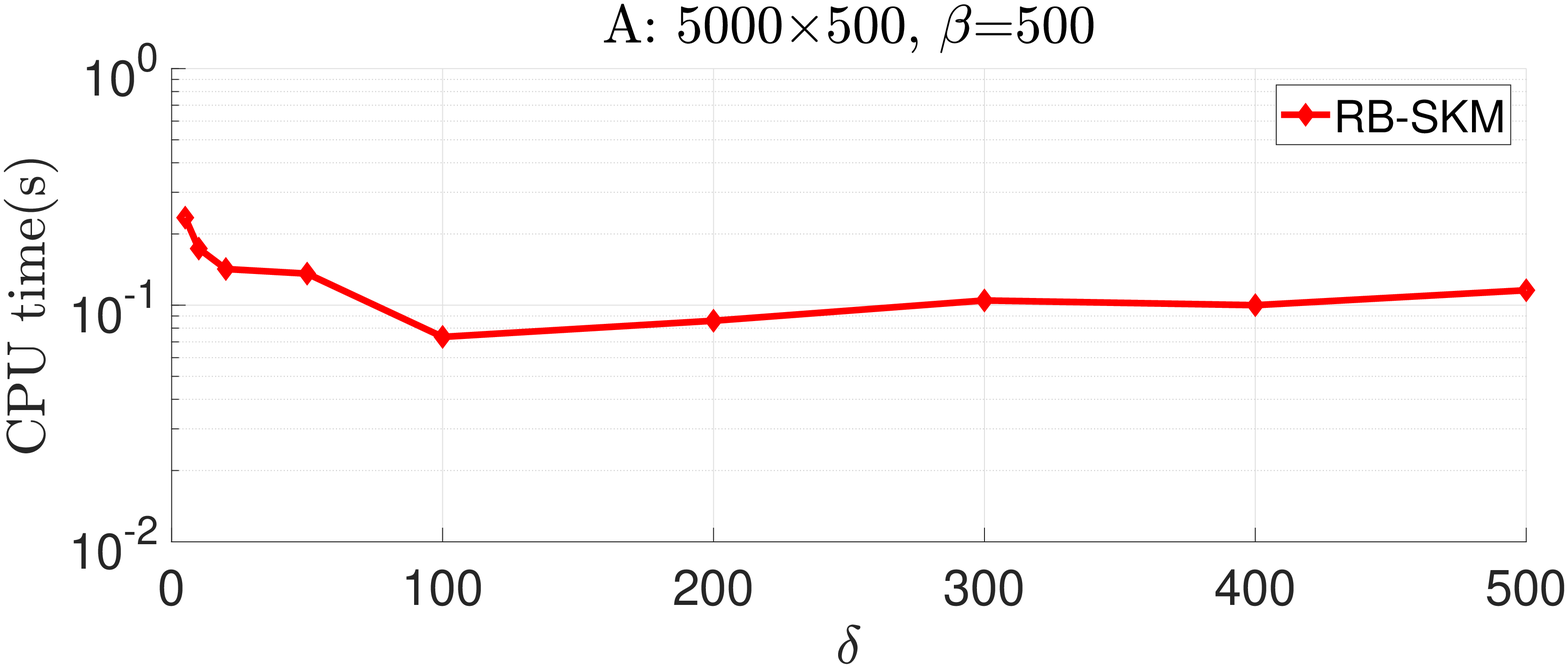}
 \end{center}
\caption{CPU time with varying values of $\beta$ (left) and $\delta$ (right) for the RB-SKM method.}\label{fig_for different parameters}
\end{figure}

\subsection{Comparison of the BSKM2 and RB-SKM methods for realistic practical problems}
\label{sec: Comparison of the BSKM2 and RB-SKM methods}

In \cite{liu2021greedy}, Liu and Gu verified the superiority of the BEM method \cite{liu2021greedy} over the RBK \cite{needell2014paved}, GBK \cite{Niu2020}, GRBK \cite{liu2021greedy} and Motzkin \cite{Motzkin54} methods through a large number of experiments, and compared to the BEM method, the BSKM1 method \cite{zhang2021block} requires less computing time because the latter only needs to pass over the data once to obtain the same iteration index set, while the former needs twice. Furthermore, extensive experiments in \cite{zhang2021block} show that the BSKM1 and BSKM2 methods have almost the same performance, but the former cannot control the size of the index set. 
So, we only need to compare our method with the BSKM2 method in the following. In fact, the left two subgraphs in Figure \ref{fig_for different parameters} have already implied that the RB-SKM method is superior to the BSKM1 method in terms of computing time when choosing  appropriate $\beta$. This is because the RB-SKM method can be approximated as the BSKM1 method when $\beta=m$ shown in Table \ref{Al_rm1:tab1}.

For the BSKM2 and RB-SKM methods, the CGLS method \cite{bjorck1996numerical} is used to avoid calculating the Moore-Penrose pseudoinverse. Thus, from Algorithm \ref{Al: The RB-SKM method}, we can obtain that the RB-SKM method in each iteration needs around $2mn+n+2\beta+ (2\text{nnz}(A_{\mathcal{I}_{k}})+3n+2\delta)\text{IT}_1+2n\delta$ operation counts, where $\text{IT}_1$ is the iteration numbers of the CGLS algorithm. Similarly, we can get that the operation counts needed for the BSKM2 method in each iteration is around $2mn+n+2m+ (2\text{nnz}(A_{\mathcal{J}_{k}})+3n+2\delta)\text{IT}_1+2n\delta$, where $\mathcal{J}_{k}$ is the iteration index set in each iteration. Therefore, our RB-SKM method needs less operation counts compared with the BSKM2 method in each iteration as $\beta\leq m$.

We use realistic practical datasets including the coefficient matrix $A$ and the corresponding right-hand side vector $b$. 
They are all taken from \cite{Davis2011}, and 
have disparate properties, either full-rank or rank-deficient, ill or well conditioned; see details in Table \ref{tab_properties of dataset}. Numerical results are reported in Figures \ref{fig_compare_iteration_with BSKM2-1} to \ref{fig_compare_CPU with BSKM2-1}, which show that our RB-SKM method outperforms the BSKM2 method in terms of iteration numbers, operation counts and CPU time.

\begin{table}[]
\centering
   \fontsize{8}{8}\selectfont
       \caption{Details of the data sets.}
    \label{tab_properties of dataset}
    \begin{tabular}{ c| c c c c  c c}
 \hline
   dataset       &$m$               &  $n$              & Nonzeros      & Condition Number  &Full Rank &  Background            \cr \hline
   well1850      &1850              &  712              & 8755          & 1.11e+02          &true      &  Least Squares Problem \cr
   Maragal\_4    &1964              &  1034             & 26719         & 6.12e+33          &false     &  Least Squares Problem \cr
   illc1033      &1033              &  320              & 4719          & 1.90e+04          &true      &  Least Squares Problem \cr
   lpi\_gran     &2658              &  2525             & 20111         & 1.16e+20          &false     &  Linear Programming Problem\cr\hline
\end{tabular}
\end{table}
\begin{figure}[ht]
 \begin{center}
 \includegraphics[height=3.2cm,width=3.7cm ]{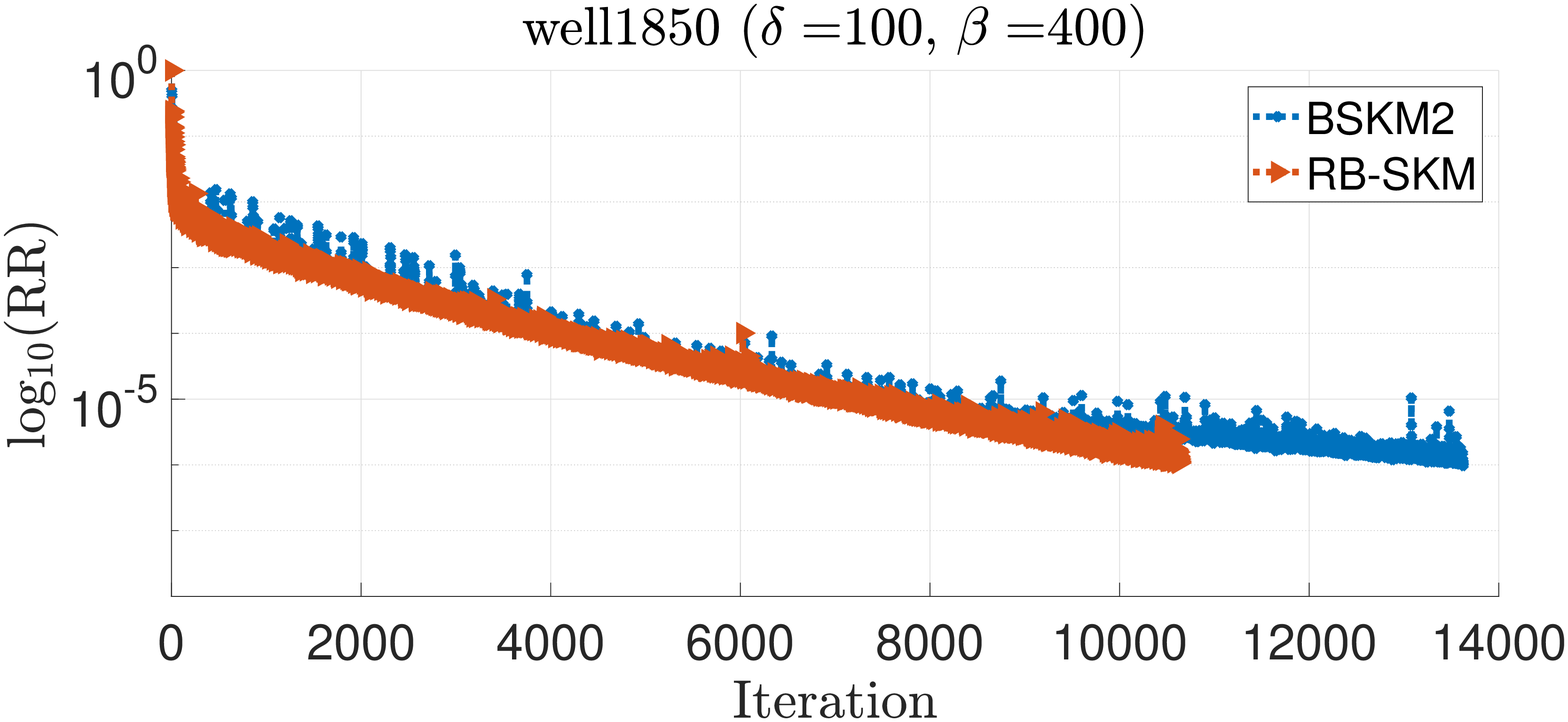}
\includegraphics [height=3.2cm,width=3.7cm ]{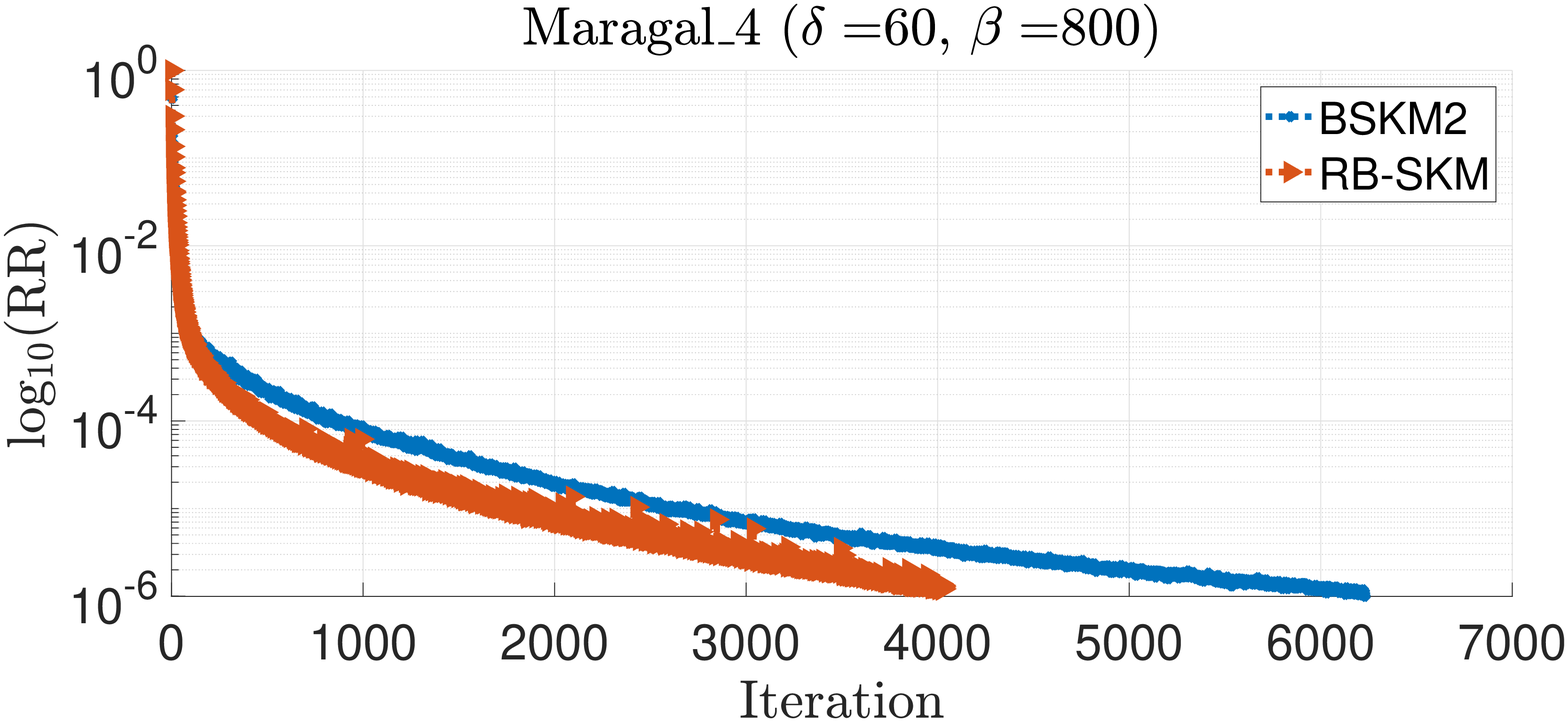}
\includegraphics [height=3.2cm,width=3.7cm ]{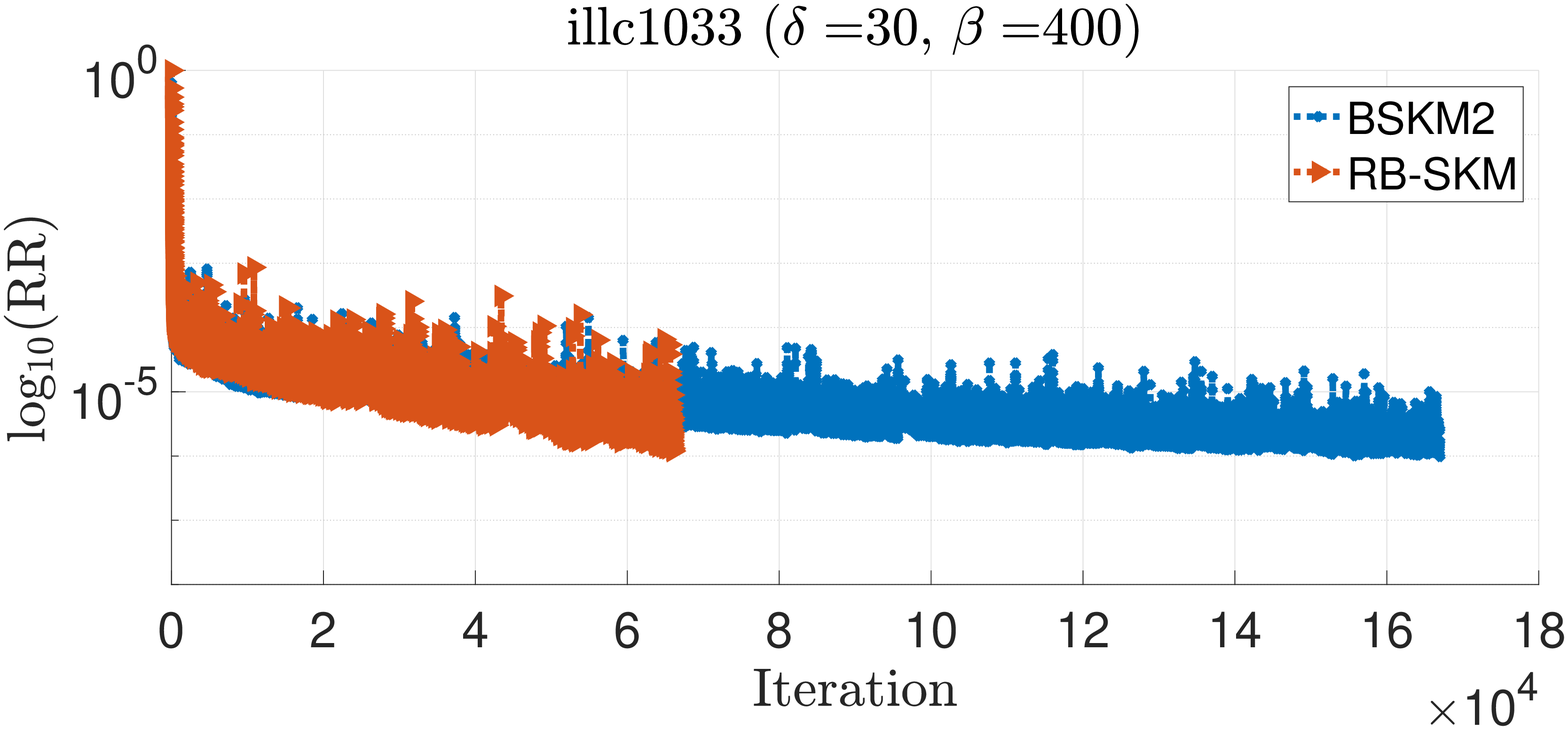}
\includegraphics [height=3.2cm,width=3.7cm ]{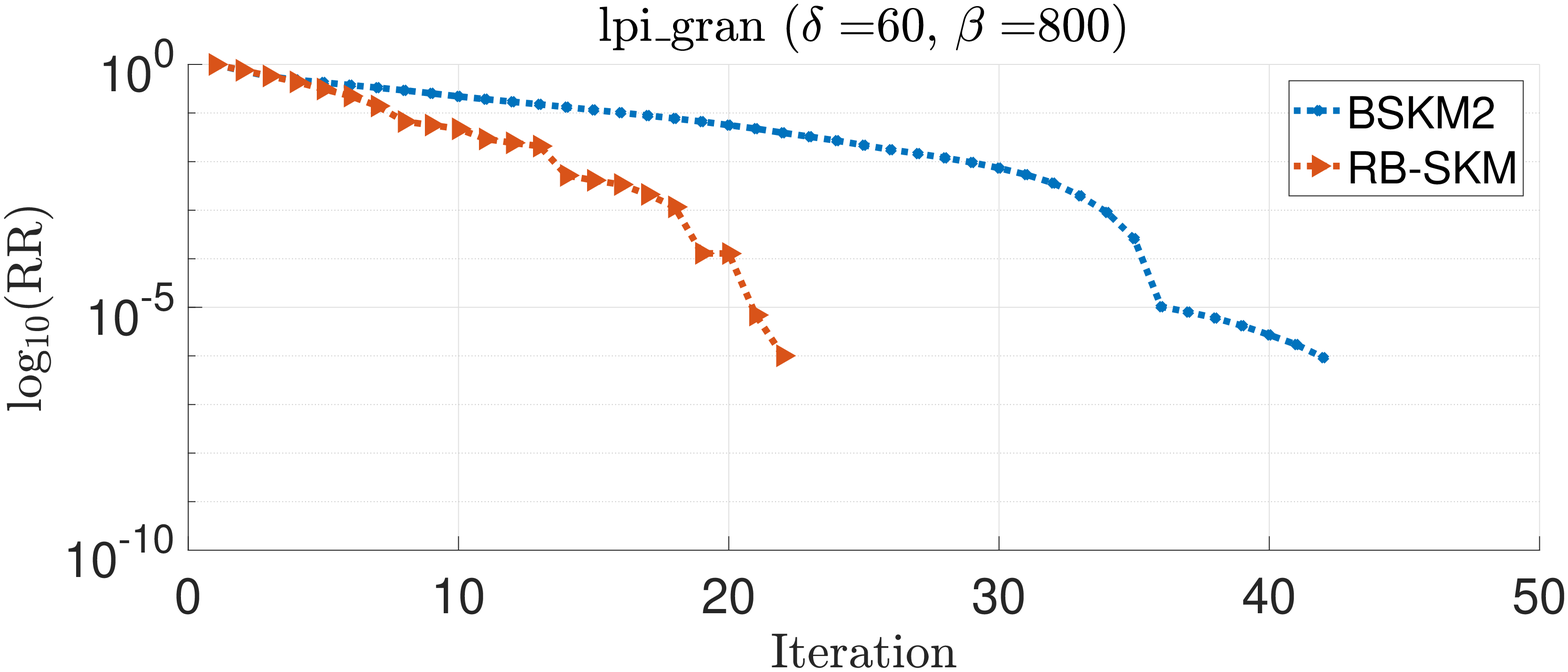}
 \end{center}
\caption{Relative residual versus iteration numbers for the BSKM2 and RB-SKM methods.}\label{fig_compare_iteration_with BSKM2-1}
\end{figure}
 \begin{figure}[ht]
 \begin{center}
\includegraphics [height=3.2cm,width=3.7cm ]{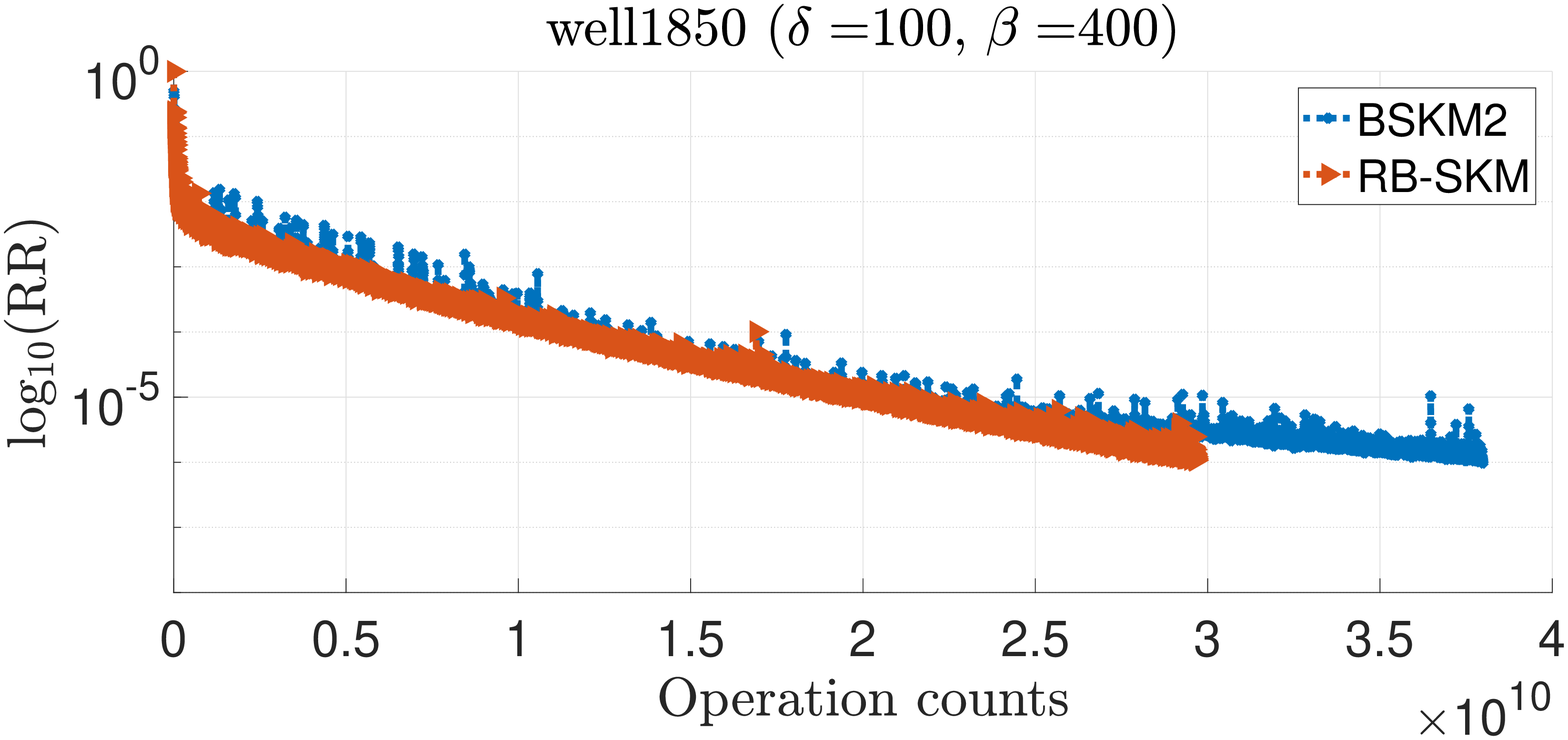}
\includegraphics [height=3.2cm,width=3.7cm ]{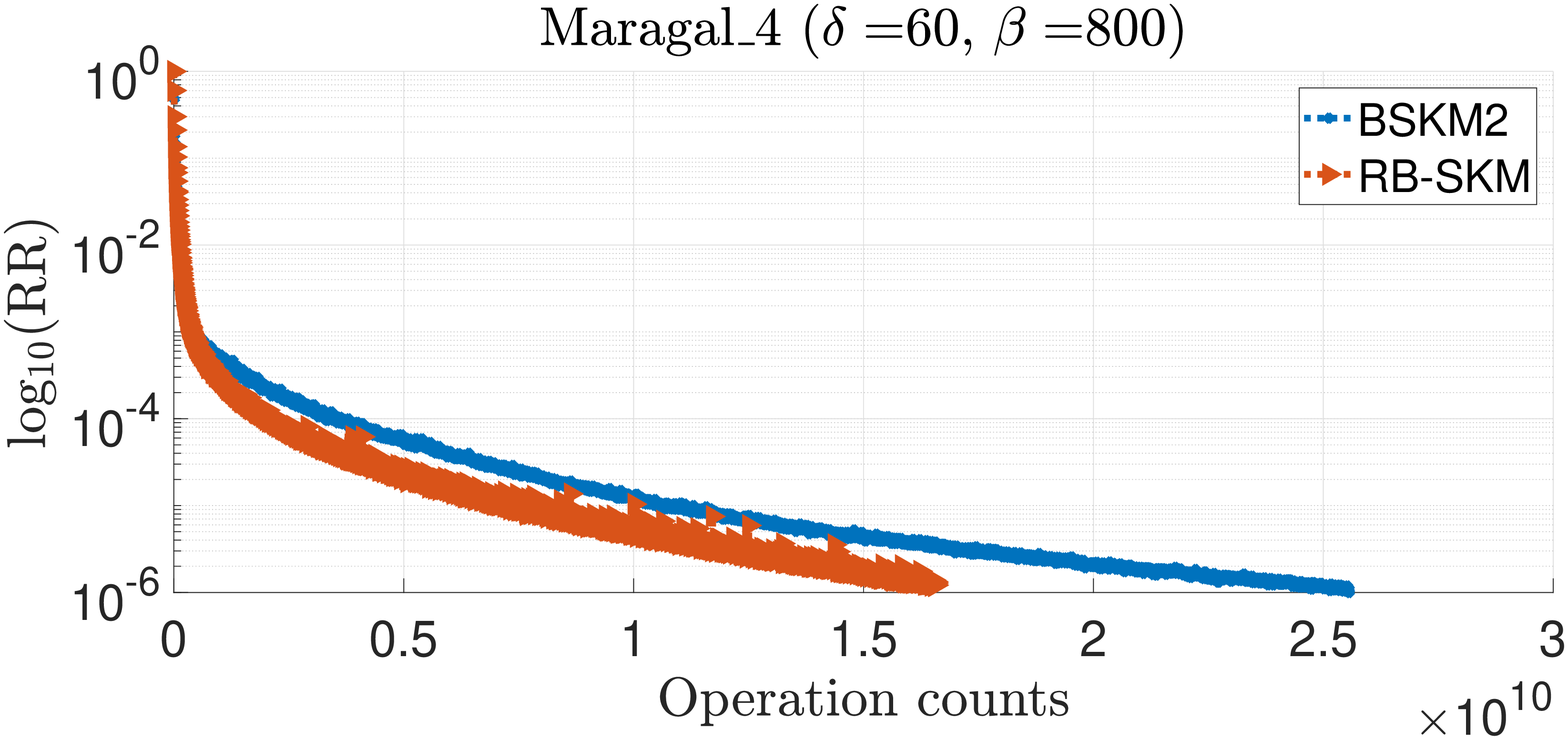}
\includegraphics [height=3.2cm,width=3.7cm ]{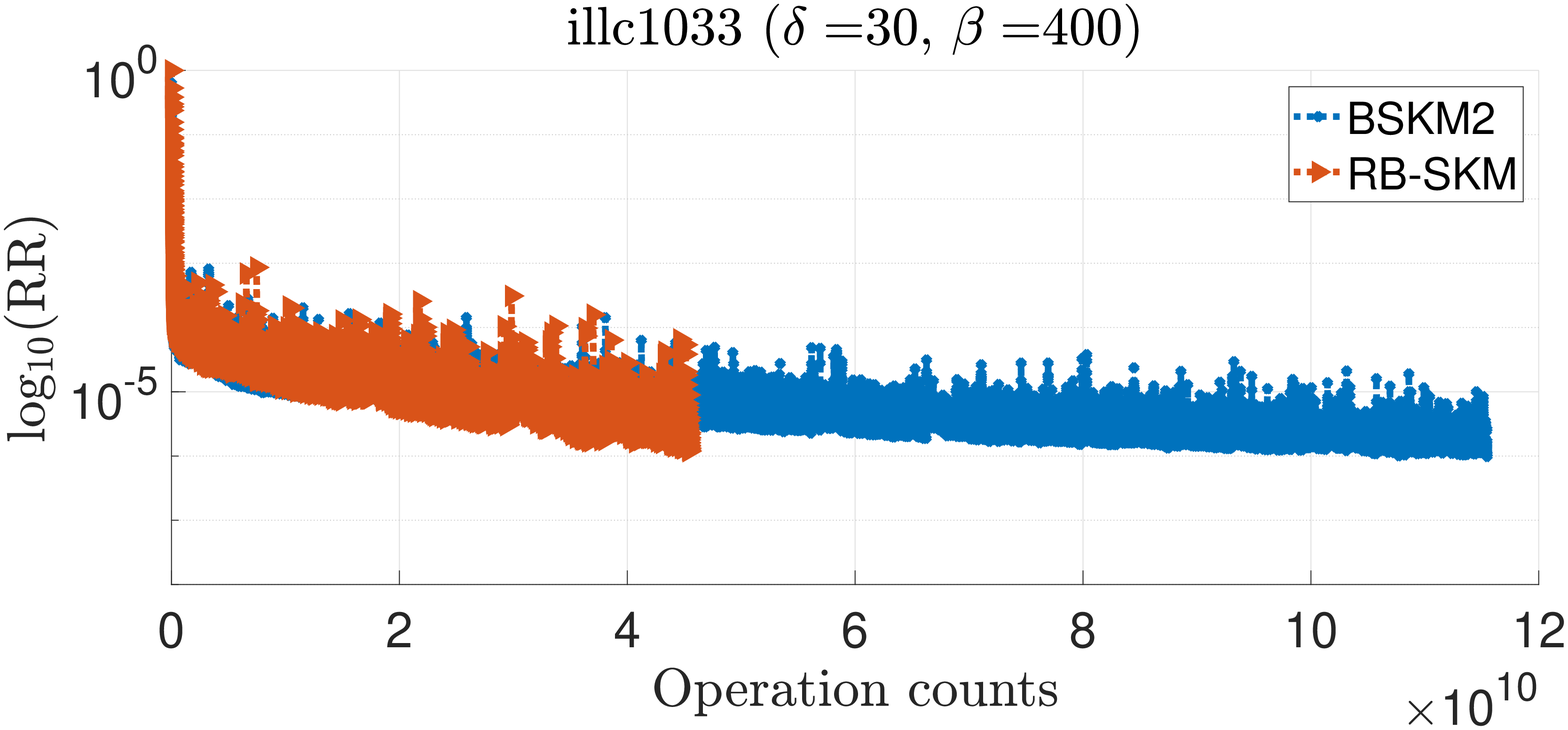}
\includegraphics [height=3.2cm,width=3.7cm ]{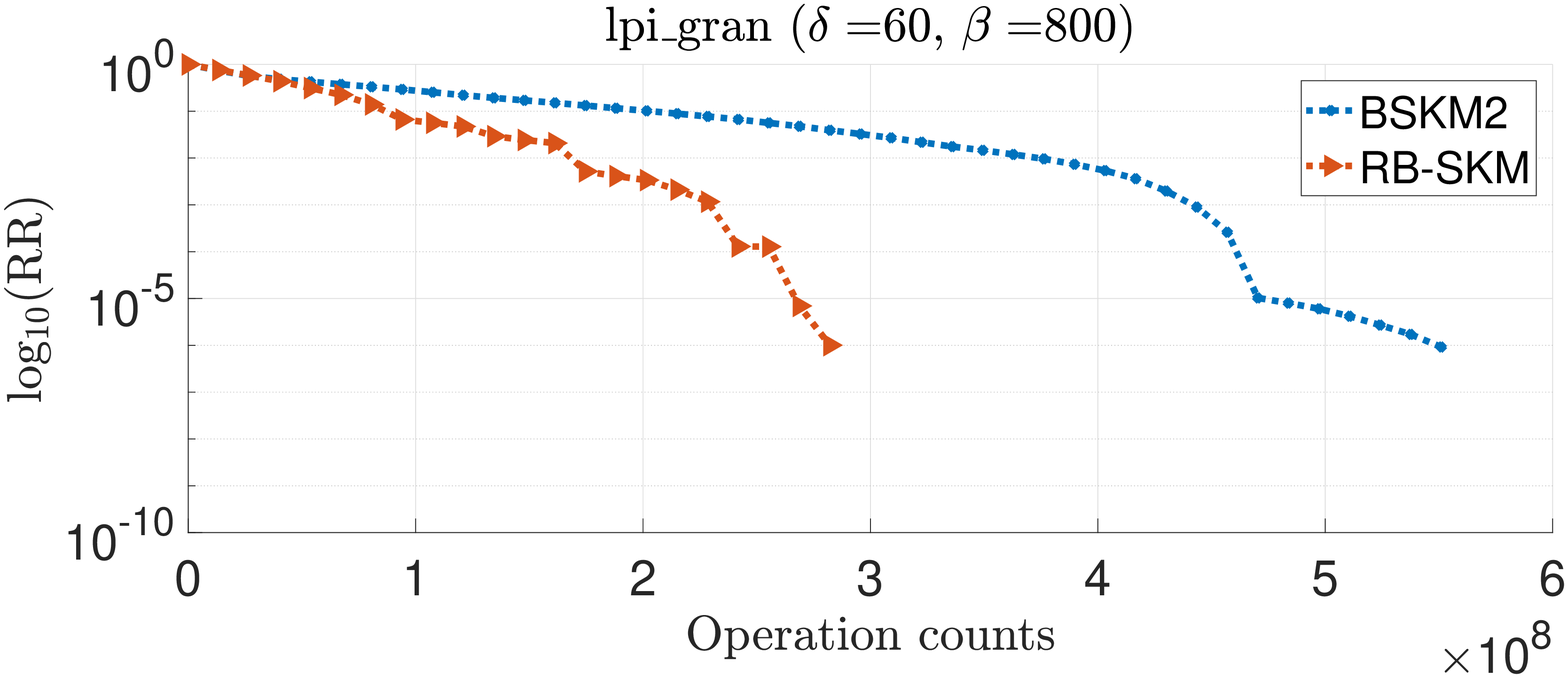}
 \end{center}
\caption{Relative residual versus operation counts for the BSKM2 and RB-SKM methods.}\label{fig_compare_Operation with BSKM2-1}
\end{figure}
\begin{figure}[ht]
 \begin{center}
\includegraphics [height=3.2cm,width=3.7cm ]{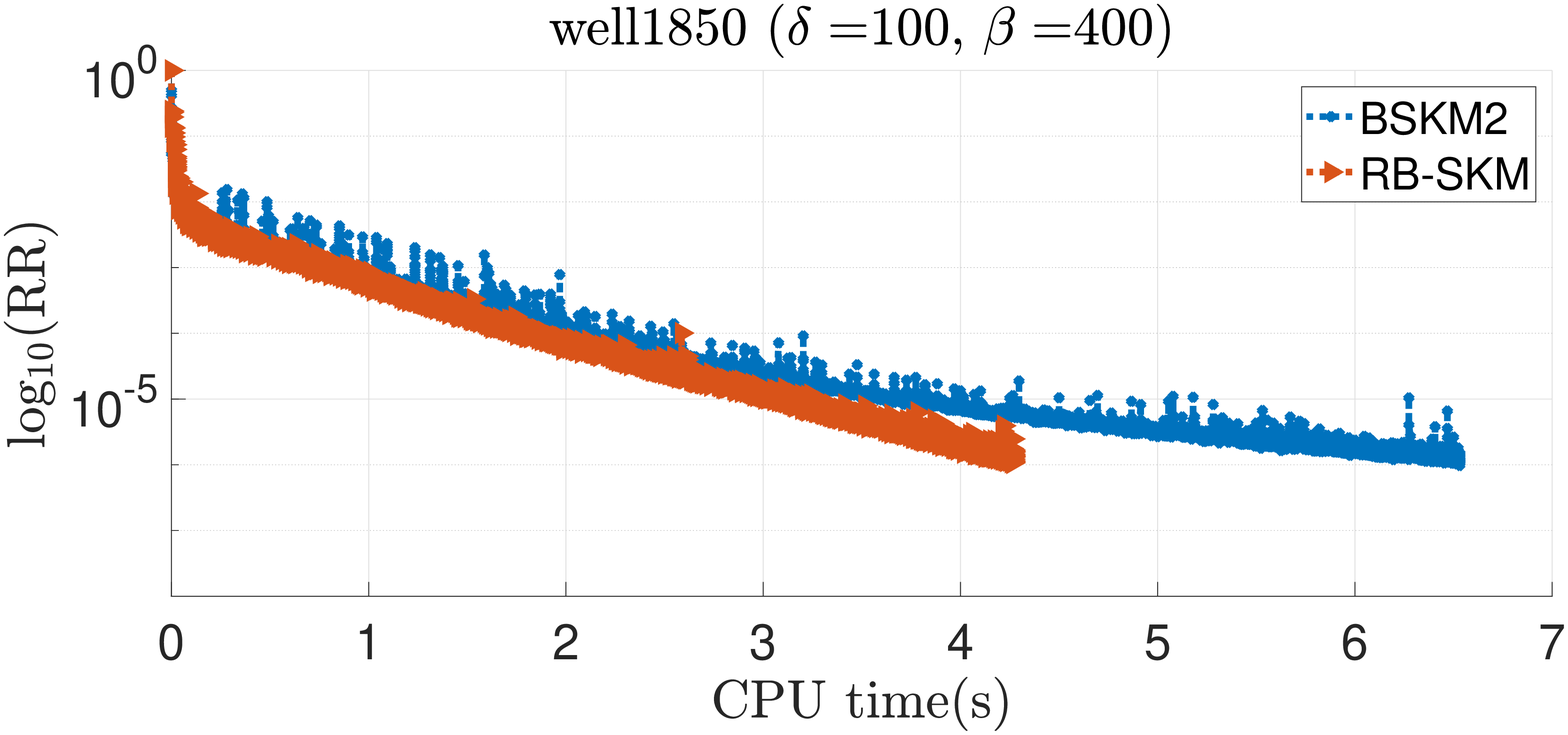}
\includegraphics [height=3.2cm,width=3.7cm ]{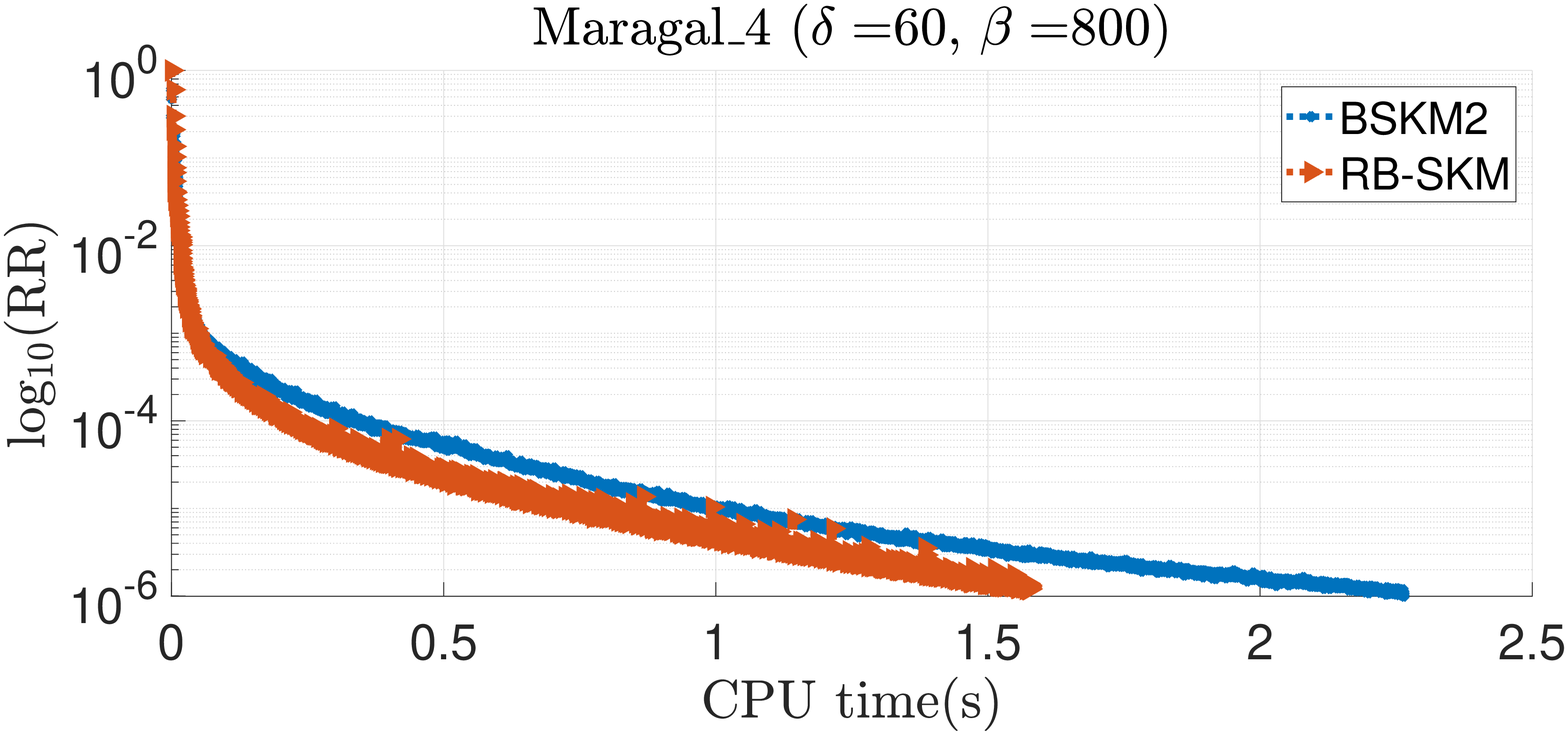}
\includegraphics [height=3.2cm,width=3.7cm ]{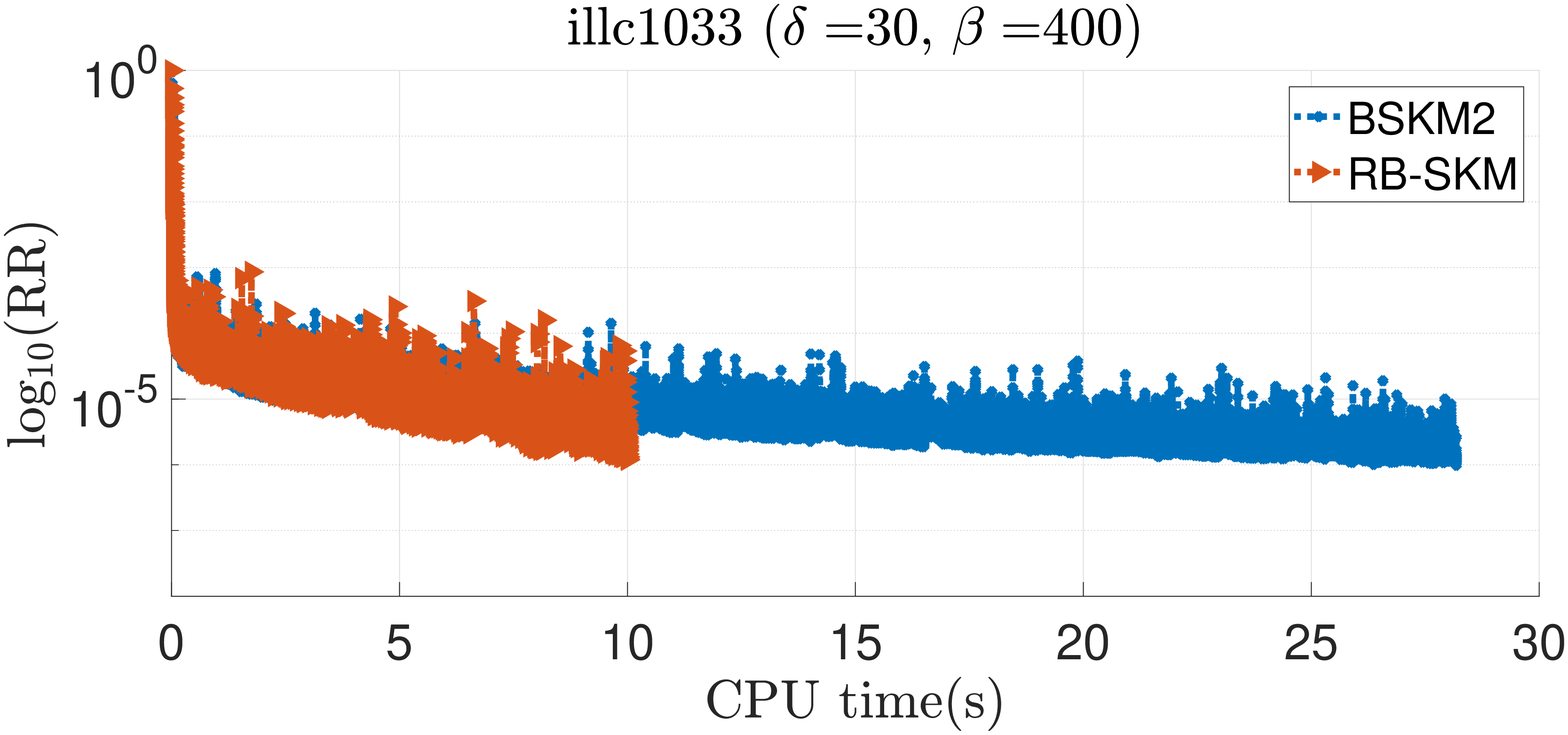}
\includegraphics [height=3.2cm,width=3.7cm ]{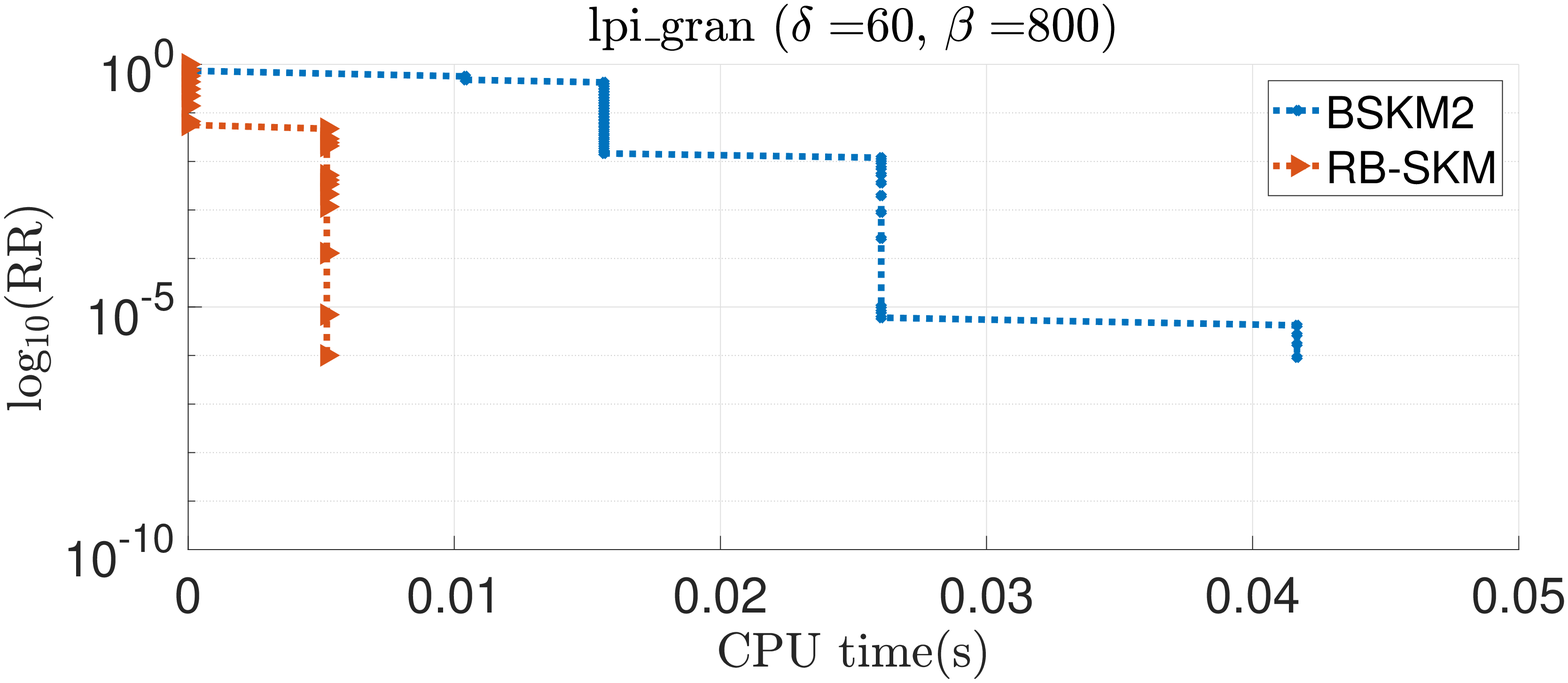}
 \end{center}
\caption{Relative residual versus CPU time for the BSKM2 and RB-SKM methods.}\label{fig_compare_CPU with BSKM2-1}
\end{figure}

\section{Concluding remarks}
\label{sec:conclusions}

This paper mainly proposes the RB-SKM method for solving consistent linear systems. It  
behaves quite well in numerical experiments. So, it is interesting to 
generalize this method for solving other problems such as the inconsistent problems \cite{zouzias2013randomized}, the ridge regression problems \cite{hefny2017rows}, the feasibility problems \cite{de2017sampling}, etc. In addition, how to determine the optimal parameters $\beta$ and $\delta$ is currently unavailable and can be regarded as a 
future work.



\bibliography{mybibfile}

\begin{thebibliography}{10}
\expandafter\ifx\csname url\endcsname\relax
  \def\url#1{\texttt{#1}}\fi
\expandafter\ifx\csname urlprefix\endcsname\relax\def\urlprefix{URL }\fi
\expandafter\ifx\csname href\endcsname\relax
  \def\href#1#2{#2} \def\path#1{#1}\fi

\bibitem{kaczmarz1}
S.~Kaczmarz, Angen\"aherte aufl\"osung von systemen linearer gleichungen, Bull.
  Int. Acad. Pol. Sci. Lett. A 35 (1937) 355--357.

\bibitem{Strohmer2009}
T.~Strohmer, R.~Vershynin, {A randomized Kaczmarz algorithm with exponential
  convergence}, J. Fourier Anal. Appl. 15 (2009) 262--278.

\bibitem{Completion2015}
A.~Ma, D.~Needell, A.~Ramdas, {Convergence properties of the randomized
  extended Gauss-Seidel and Kaczmarz methods}, SIAM J. Matrix Anal. Appl. 36
  (2015) 1590--1604.

\bibitem{gower2015randomized}
R.~M. Gower, P.~Richt{\'a}rik, Randomized iterative methods for linear systems,
  SIAM J. Matrix Anal. Appl. 36 (2015) 1660--1690.

\bibitem{Eldar2011}
Y.~Eldar, D.~Needell, {Acceleration of randomized Kaczmarz method via the
  Johnson-Lindenstrauss lemma}, Numer. Algor. 58 (2011) 163--177.

\bibitem{liu2016accelerated}
J.~Liu, S.~J. Wright, {An accelerated randomized Kaczmarz algorithm}, Math
  Comp. 85 (2016) 153--178.

\bibitem{jiao2017preasymptotic}
Y.~L. Jiao, B.~T. Jin, X.~L. Lu, {Preasymptotic convergence of randomized
  Kaczmarz method}, Inverse Problems 33 (2017) 125012.

\bibitem{de2017sampling}
J.~A. De~Loera, J.~Haddock, D.~Needell, {A sampling Kaczmarz-Motzkin algorithm
  for linear feasibility}, SIAM J. Sci. Comput. 39 (2017) S66--S87.

\bibitem{bai2018greedy}
Z.~Z. Bai, W.~T. Wu, {On greedy randomized Kaczmarz method for solving large
  sparse linear systems}, SIAM J. Sci. Comput. 40 (2018) A592--A606.

\bibitem{gower2021adaptive}
R.~M. Gower, D.~Molitor, J.~Moorman, D.~Needell, On adaptive sketch-and-project
  for solving linear systems, SIAM J. Matrix Anal. Appl. 42 (2021) 954--989.

\bibitem{needell2014paved}
D.~Needell, J.~A. Tropp, {Paved with good intentions: Analysis of a randomized
  block Kaczmarz method}, Linear Algebra Appl. 441 (2014) 199--221.

\bibitem{Niu2020}
Y.~Q. Niu, B.~Zheng, {A greedy block Kaczmarz algorithm for solving large-scale
  linear systems}, Appl. Math. Lett. 104 (2020) 106294.

\bibitem{liu2021greedy}
Y.~Liu, C.~Q. Gu, {On greedy randomized block Kaczmarz method for consistent
  linear systems}, Linear Algebra Appl. 616 (2021) 178--200.

\bibitem{zhang2021block}
Y.~J. Zhang, H.~Y. Li, {Block sampling Kaczmarz-Motzkin methods for consistent
  linear systems}, Calcolo 58 (2021) 39.

\bibitem{Zhang2022MK}
Y.~J. Zhang, H.~Y. Li, {Greedy Motzkin-Kaczmarz methods for solving linear
  systems}, Numer. Linear Algebra Appl. 29 (2022) e2429.

\bibitem{horn2012matrix}
R.~A. Horn, C.~R. Johnson, Matrix Analysis, Cambridge University Press,
  Cambridge, 2012.

\bibitem{Motzkin54}
T.~S. Motzkin, I.~J. Schoenberg, The relaxation method for linear inequalities,
  Canad. J. Math. 6 (1954) 393--404.

\bibitem{necoara2019faster}
I.~Necoara, {Faster randomized block Kaczmarz algorithms}, SIAM J. Matrix Anal.
  Appl. 40 (2019) 1425--1452.

\bibitem{du2020randomized}
K.~Du, W.~T. Si, X.~H. Sun, {Randomized extended average block Kaczmarz for
  solving least squares}, SIAM J. Sci. Comput. 42 (2020) A3541--A3559.

\bibitem{du2021doubly}
K.~Du, X.~H. Sun, {A doubly stochastic block Gauss-Seidel algorithm for solving
  linear equations}, Appl. Math. Comput. 408 (2021) 126373.

\bibitem{chen2022fast}
J.~Q. Chen, Z.~D. Huang, {On a fast deterministic block Kaczmarz method for
  solving large-scale linear systems}, Numer. Algor. 89 (2022) 1007--1029.

\bibitem{haddock2021greed}
J.~Haddock, A.~Ma, {Greed works: An improved analysis of sampling
  Kaczmarz-Motzkin}, SIAM J. Math. Data Sci. 3 (2021) 342--368.

\bibitem{haddock2019motzkin}
J.~Haddock, D.~Needell, {On Motzkin's method for inconsistent linear systems},
  BIT Numer. Math. 59 (2019) 387--401.

\bibitem{bjorck1996numerical}
{\AA}.~Bj{\"o}rck, Numerical Methods for Least Squares Problems, SIAM,
  Philadelphia, 1996.

\bibitem{Davis2011}
T.~A. Davis, Y.~F. Hu, {The university of Florida sparse matrix collection},
  ACM. Trans. Math. Softw. 38 (2011) 1--25.

\bibitem{zouzias2013randomized}
A.~Zouzias, N.~M. Freris, {Randomized extended Kaczmarz for solving least
  squares}, SIAM J. Matrix Anal. Appl. 34 (2013) 773--793.

\bibitem{hefny2017rows}
A.~Hefny, D.~Needell, A.~Ramdas, {Rows versus columns: Randomized Kaczmarz or
  Gauss-Seidel for ridge regression}, SIAM J. Sci. Comput. 39 (2017)
  S528--S542.

\end{thebibliography}

\end{document}